\documentclass[a4paper,11pt]{amsart}

\usepackage[utf8]{inputenc}		
\usepackage[T1]{fontenc}
\usepackage[english]{babel}

\usepackage{amsfonts}			
\usepackage{amsmath}
\usepackage{amssymb}
\usepackage{amsthm}

\usepackage[normalem]{ulem}

\usepackage{graphicx}			
\usepackage{tikz}
\usetikzlibrary{cd}

\usepackage{hyperref}			
\hypersetup{colorlinks=false, urlcolor=black, linkcolor=black}
\usepackage{cleveref}

\usepackage{enumitem}			

\usepackage{color}				
\usepackage{float}

\usepackage{thmtools}
\usepackage{thm-restate}

\usepackage{mathrsfs}

\newcommand{\Z}{\mathbb{Z}}						
\newcommand{\R}{\mathbb{R}}						
\newcommand{\C}{\mathbb{C}}						

\renewcommand{\S}{\mathbb{S}}					
\newcommand{\B}{\mathbb{B}}
\renewcommand{\H}{\mathbb{H}}					

\newcommand{\eps}{\varepsilon}					

\newcommand{\dd}								
{\mathop{}\!\mathrm{d}}						
\newcommand{\ddn}[1]							
{\mathop{}\!\mathrm{d^{#1}}}

\newcommand{\abs}[1]							
{\left| #1 \right|}
\newcommand{\smallabs}[1]						
{\lvert #1 \rvert}	
\newcommand{\norm}[1]							
{\left\lVert #1 \right\rVert}	
\newcommand{\smallnorm}[1]						
{\lVert #1 \rVert}						
\newcommand{\ip}[2]								
{\left< #1 , #2 \right>}

\DeclareMathOperator{\vol}{vol}					
\DeclareMathOperator{\spt}{spt}					

\DeclareMathOperator{\dist}{dist}

\let\Im\relax

\DeclareMathOperator{\Im}{Im}					

\newcommand{\loc}{\mathrm{loc}}

\newcommand{\cH}{\mathcal{H}}

\newcommand{\cX}{\mathcal{X}}

\renewcommand{\phi}{\varphi}

\def\Xint#1{\mathchoice
	{\XXint\displaystyle\textstyle{#1}}%
	{\XXint\textstyle\scriptstyle{#1}}%
	{\XXint\scriptstyle\scriptscriptstyle{#1}}%
	{\XXint\scriptscriptstyle\scriptscriptstyle{#1}}%
	\!\int}
\def\XXint#1#2#3{{\setbox0=\hbox{$#1{#2#3}{\int}$}
		\vcenter{\hbox{$#2#3$}}\kern-.5\wd0}}

\def\dashint{\Xint-}

\newtheorem{thm}{Theorem}[section]{\bf}{\it}
\newtheorem{lem}[thm]{Lemma}
\newtheorem{prop}[thm]{Proposition}
\newtheorem{cor}[thm]{Corollary}

{\bf}{\it}

{\bf}{\it}

{\bf}{\it}

{\bf}{\it}

\theoremstyle{definition}
\newtheorem{defn}[thm]{Definition}
\newtheorem{ex}[thm]{Example}
\newtheorem{example}[thm]{Example}
\newtheorem{setting}[thm]{Setting}

\theoremstyle{remark}

\numberwithin{equation}{section}

\begin{document}
\title[Values of finite distortion: Reshetnyak and Liouville]{Values of finite distortion: Reshetnyak's theorem, the Liouville Theorem, and the Lusin (N) -property}

\author[I. Kangasniemi]{Ilmari Kangasniemi}
\address{Department of Mathematics and Statistics, 
    P.O. Box 35 (MaD), FI-40014 University of Jyväskylä, Finland.
}
\email{ilmari.k.kangasniemi@jyu.fi}

\author[J. Onninen]{Jani Onninen}
\address{Department of Mathematics, Syracuse University, Syracuse,
    NY 13244, USA, and  Department of Mathematics and Statistics, P.O. Box 35 (MaD) FI-40014 University of Jyv\"askyl\"a, Finland
}
\email{jkonnine@syr.edu}

\author[Y. Zhu]{Yizhe Zhu}
\address{Department of Mathematics and Statistics, 
    P.O. Box 35 (MaD), FI-40014 University of Jyväskylä, Finland.
}
\email{yizhe.y.zhu@jyu.fi}

\date{\today}
\thanks{I.K.\ was supported by the Research Council of Finland grant \#369659. J.O.\ was supported by the U.S.\ National Science Foundation grant DMS-2453853. Y.Z.\ was supported by the Research Council of Finland grant \#334014.}

\begin{abstract}
    Let $\Omega \subset \R^n$ be a domain and $f \in W^{1,n}_{\loc} (\Omega,\R^n) $. We say that $f$ has a \emph{value of finite distortion} at $y_0 \in \R^n$ if there exist measurable functions $K \colon \Omega \to [0,\infty) $ and $\Sigma \in L^1_{\loc} (\Omega)$ such that
    \[
        \abs{Df(x)}^n \le K(x) \det Df (x) + \Sigma(x)  \abs{f(x)-y_0}^n
        \quad \text{for a.e. } x \in \Omega.
    \]
    This notion unifies the classical theory of mappings of finite distortion with the recently introduced theory of quasiregular values. Under sharp integrability assumptions on $K$ and $\Sigma$, we establish single-value analogues of Reshetnyak’s theorem and the Liouville theorem. 
     We also prove that mappings satisfying a more general distortion inequality with defect preserve sets of Lebesgue measure zero.
\end{abstract}

\maketitle

\section{Introduction}

Let $\Omega \subset \R^n$, $n\ge2$, be a domain. A mapping $f\colon \Omega \to \R^n$ in $W^{1,n}_{\loc}(\Omega, \R^n)$ is said to
have \emph{finite distortion} if there exists a measurable function $K\colon \Omega\to[1,\infty)$ such that
\begin{equation}\label{eq: finite distortion}
    |Df(x)|^n\le K(x)J_f(x)
\end{equation}
for almost every $x\in\Omega$. Here $|Df(x)|$ denotes the operator norm
of the weak differential of $f$ at $x$, and $J_f(x)$ its Jacobian
determinant.

The special case in which $K$ is bounded gives rise to the class of
\emph{mappings of bounded distortion}, or \emph{quasiregular mappings}.
A fundamental theorem of Reshetnyak~\cite{reshetni_ak1989space,reshetnyak1968condition}
asserts that every quasiregular mapping is continuous and either
constant or open, discrete, and sense-preserving. In this way,
quasiregular mappings provide a natural higher-dimensional analogue of
holomorphic functions.

For general mappings of finite distortion, continuity of
$W^{1,n}_{\loc}$-mappings was established by Gol'dshtein and
Vodop'yanov~\cite{vodop1976quasiconformal}, and later extended to
Orlicz--Sobolev settings by Iwaniec, Koskela, and
Onninen~\cite{iwaniec2001mappings}. The corresponding openness and
discreteness theory is substantially more delicate. In the planar case,
Iwaniec and \v Sver\'ak~\cite{iwaniec1993mappings} proved that every
nonconstant mapping of finite distortion with
$K\in L^1_{\loc}(\Omega)$ is open and discrete. In dimensions
$n\ge3$, Manfredi and Villamor~\cite{manfredi1995mappings} showed
that the same conclusion holds under the assumption
$K\in L^p_{\loc}(\Omega)$ with $p>n-1$. The exponent $n-1$ is sharp,
as demonstrated by Ball's example~\cite{ball1981global}. At the
borderline exponent $p=n-1$, Hencl and
Rajala~\cite{henclrajala2013} constructed counterexamples showing
that, in general, the assumption $p>n-1$ cannot be removed when
$n\ge3$.

Another central result in the theory is the Liouville theorem:
bounded quasiregular mappings
$f\colon\R^n\to\R^n$ are constant
\cite{reshetnyak1967liouville}. In the planar case, this follows
immediately from the Sto\"ilow factorization theorem together with the
classical Liouville theorem of complex analysis. Martio, Rickman, and
V\"ais\"al\"a~\cite{manrro1970distortion} later proved that every
$K$-quasiregular mapping satisfying a suitable slow growth condition
must also be constant. Since then, numerous refinements and analogues
have been established, including Rickman's Picard
theorem~\cite{rickman1980picard} and various manifold-valued constancy theorems; see for instance, \cite{Varopoulos-SaloffCoste-Coulhon,troyanov2002liouville,Bonk-Heinonen_Acta,Prywes_Annals,Heikkila-Pankka_Elliptic}.

\subsection{Quasiregular values and values of finite distortion} \enlargethispage{\baselineskip}

Recently, a theory of \emph{quasiregular values} has emerged, providing
single-value analogues of several fundamental results in the theory of
quasiregular mappings. The origins of the theory trace back to a definition by Astala, Iwaniec, and Martin \cite[Section 8.5]{iwaniec2001geometric}. A systematic
framework was subsequently developed by Kangasniemi, Onninen, and
Heikkil\"a in a series of papers
\cite{kangasniemi2022heterogeneous,kangasniemi2024correction,
kangasniemi2025single,kangasniemi2024quasiregular,
kangasniemi2024linear,heikkila2025quasiregular}.

In this article, we consider an analogous generalization for  mappings of finite distortion.

\begin{defn}
    Let $\Omega\subset\mathbb{R}^n$ be a domain, let $y_0\in\mathbb{R}^n$, let $K \colon \Omega \to [1, \infty)$ be measurable, and let $\Sigma\in L^{1}_\loc(\Omega)$. Suppose that $f\in W^{1,n}_{\rm loc}(\Omega,\mathbb{R}^n)$. We say that $f$ has a \emph{value of finite distortion} at $y_0$ with data $(K, \Sigma)$ if
    \begin{align}\label{K,Sigma-quasiregular}
        |Df(x)|^n\le K(x)J_f(x)+\Sigma(x)|f-y_0|^n
    \end{align}
    for a.e.\ $x\in\Omega$. Alternatively, we say that $f$ satisfies the \emph{$(K, \Sigma)$-distortion inequality with defect} if
    \begin{align}\label{K,Sigma-distortion}
        |Df(x)|^n\le K(x)J_f(x)+\Sigma(x)
    \end{align}
    for a.e.\ $x \in \Omega$.
\end{defn}

Conditions of this type first appeared in \cite{dolevzalova2024mappings}.  Sharp continuity results for such conditions were recently established in \cite{kangasniemi2025continuity}:  If $K \in L^p_\loc(\Omega)$ and $\Sigma/K \in L^q_\loc(\Omega)$ with $p, q \in [1, \infty]$, then the solutions $f \in W^{1,n}_{\loc}(\Omega,\mathbb{R}^n)$ of \eqref{K,Sigma-quasiregular} or \eqref{K,Sigma-distortion} are continuous precisely when $p^{-1} + q^{-1} < 1$. Counterexamples in the complementary range  $p^{-1} + q^{-1} \ge 1$ were constructed in \cite{dolevzalova2024mappings}.

\subsection{Reshetnyak's theorem for values of finite distortion}
Our first main result is a Reshetnyak-type theorem for mappings
possessing a value of finite distortion. The corresponding theorem for quasiregular values was shown in
\cite[Theorem~1.2]{kangasniemi2025single}.  A generalization has since been shown for variable $K$ by D.\ Zhong \cite{Zhong_ReshetnyakGeneralization}, though under a relatively strong assumption on the local behavior of $K$ that for instance entirely disallows radial singularities.

\begin{thm}\label{open and discrete thm}
    Let $\Omega\subset\mathbb{R}^n$ be a domain, let $y_0\in\mathbb{R}^n$, and let $p, q \in [1, \infty]$. Suppose that
    $f\in W^{1,n}_{\rm loc}(\Omega,\mathbb{R}^n)$ has a value of finite distortion at $y_0\in\mathbb{R}^n$ with data $(K,\Sigma)$, where $K:\Omega\to[1,\infty)$ and $\Sigma:\Omega\to[0,\infty)$ are measurable functions such that
    \begin{equation}\label{K,Sigma condition}
        K\in L^p_{\rm loc}(\Omega),
        \quad{\rm and}\quad 
        \frac{\Sigma}{K}\in L^{q}_{\rm loc}(\Omega).
    \end{equation}
    If $p > n-1$ and $p^{-1} + q^{-1} < 1$, then either $f\equiv y_0$ a.e.\ in $\Omega$, or the continuous representative of $f$ satisfies the following conditions:
    \begin{itemize}[noitemsep]
        \item $f^{-1}\{y_0\}$ is a discrete subset of $\Omega$,
        \item at every $x_0\in f^{-1}\{y_0\}$ the local index $i(x_0,f)$ is positive, and
        \item for every neighborhood $U$ of a point $x_0\in f^{-1}\{y_0\}$, we have $y_0\in  {\rm int}\ f(U)$.
    \end{itemize}
\end{thm}

The condition $p^{-1}+q^{-1}<1$ is essential. Indeed, solutions of
\eqref{K,Sigma-quasiregular} need not even admit continuous
representatives when this assumption fails. Moreover, even continuous solutions with $p^{-1}+q^{-1}\ge1$ may fail to be open or sense-preserving at points of $f^{-1}\{y_0\}$; see
Example~\ref{ex:planar_counterexample}.
We also note that when $p^{-1} + q^{-1} < 1$, \eqref{K,Sigma condition} yields that $\Sigma \in L^r_\loc(\Omega)$ with $r = (p^{-1} + q^{-1})^{-1} > 1$. Such higher integrability of $\Sigma$ is necessary, as demonstrated by the example given in \cite[Example 8.1]{kangasniemi2025single}. 

The assumption that $K\in L_{\loc}^p (\Omega)$ with $p>n-1$ is also necessary. For $n \ge 3$, this is already required in the classical case  $\Sigma \equiv 0$; see~\cite{henclrajala2013}. For $n = 2$, although Reshetnyak's theorem remains valid for mappings of finite distortion when $p = 1$, the fundamental difference in our setting is that if $p = 1$, then there exists no $q \in [1, \infty]$ with $p^{-1} + q^{-1} < 1$; for a counterexample with $p = 1$ and $q = \infty$, see again Example \ref{ex:planar_counterexample}.

Finally, one cannot expect Reshetnyak-type conclusions at the preimages of values  $y\neq y_0$, nor for solutions of the defect inequality \eqref{K,Sigma-distortion}. For instance, the mapping  $f(x)=(x_1,0,...,0)$, satisfies the $(K,1)$-distortion inequality with defect for any constant $K\ge 1$, yet $f$ is neither
constant nor discrete or open.  Further counterexamples are given in  \cite[Example 8.4]{kangasniemi2025single}, where a mapping $f \colon \R^n \to \R^n$ is constructed that has a $(1, \Sigma)$-quasiregular value at $y_0 = 0$ with a constant $\Sigma$, has no values of finite distortion with a locally integrable $\Sigma$ at other $y \in \R^n \setminus \{0\}$, and fails to be sense preserving at the preimages of any $y \ne 0$.

\subsection{The Liouville theorem for values of finite distortion} 
Our second main result is a Liouville-type theorem for values of finite distortion. Given a measurable function
$K\colon \R^n\to[1,\infty)$ and $p>0$, we define
\begin{equation}\label{ineq: K_n}
    \mathcal{K}_{p}
    :=\sup_{R\ge 1}\dashint_{\B(0, R)} K^{p}.
\end{equation}
Under the assumption $\mathcal K_{n-1}<\infty$, we obtain the
following  generalization of the Liouville theorem for quasiregular values
proved in \cite[Theorem~1.2]{kangasniemi2022heterogeneous}.

\begin{thm}\label{thm:liouville}
    Let $y_0\in\mathbb{R}^n$, and let $p, q \in [1, \infty]$. 
    Suppose that $f\in W^{1,n}_{\rm loc}(\R^n,\mathbb{R}^n)$ has a value of finite distortion at $y_0\in\mathbb{R}^n$ with data $(K,\Sigma)$, where $K:\R^n\to[1,\infty)$ and $\Sigma:\R^n\to[0,\infty)$ are measurable functions such that 
    \begin{equation*}
        K\in L^p_{\rm loc}(\R^n),
        \quad{\rm and}\quad 
        \frac{\Sigma}{K}\in L^{q}_{\rm loc}(\R^n)\cap L^1(\R^n).
    \end{equation*}
    If $p^{-1} + q^{-1} < 1$, $\mathcal{K}_{n-1} < \infty$, and $f$ is bounded, then the continuous representative of $f$ satisfies either $f\equiv y_0$ or $y_0 \notin f(\R^n)$.
\end{thm}

Notably, Theorem \ref{thm:liouville}, or alternatively a part of its proof, also immediately yields the following Liouville-type theorem for mappings of finite distortion.

\begin{thm}\label{thm: Liouville FDM case}
    Suppose that $f\in W^{1,n}_{\loc}(\R^n,\R^n)$ is a mapping of finite distortion with $\mathcal{K}_{n-1}<\infty$. If $f$ is bounded, then $f$ is constant. 
\end{thm}

In this case when $J_f\ge0$ almost everywhere, the argument is essentially standard. Moreover, under the slightly stronger assumption $\mathcal K_p<\infty$ for some $p>n-1$, Guo~\cite[Corollary 2.6]{Guo2014polynomial} proved that the omitted
set $\overline{\R^n}\setminus f(\R^n)$ has vanishing $n$-capacity,
even without any boundedness assumption on $f$.

The condition $\mathcal K_{n-1}<\infty$ is sharp in
Theorem~\ref{thm: Liouville FDM case}, and hence also in
Theorem~\ref{thm:liouville}. Indeed, for every $p<n-1$ there exists a
bounded nonconstant mapping of finite distortion
$f\in W^{1,n}_{\loc}(\R^n,\mathbb R^n)$ such that
$\mathcal K_p<\infty$; see
Example~\ref{exam: K_p is finite}.

The assumption $\Sigma/K\in L^1(\R^n)$ of Theorem \ref{thm:liouville} is also sharp, in the sense that the theorem is no longer valid if we change it to $\Sigma/K\in L^r(\R^n)$ for some $r > 1$. Indeed, consider the map
\[
    f \in W^{1,n}_\loc(\R^n, \R^n), \quad
    f(x) = \begin{cases}
        x & \abs{x} \le 1\\
        \frac{x}{\abs{x}} & \abs{x} > 1.
    \end{cases}
\]
This is a bounded, non-constant map with a value of finite distortion at $0$, where $K \equiv 1$ and $\Sigma(x) = \cX_{\R^n \setminus \B^n}(x)/\abs{x}^n$. We observe that $K, \Sigma \in L^{\infty}(\R^n)$ and $\Sigma/K \in L^r(\R^n)$ for every $r > 1$, but $\Sigma/K \notin L^1(\R^n)$.

\subsection{Methods and other notable results}

The proof of Theorem \ref{open and discrete thm} follows a similar outline as that of \cite[Theorem 1.2]{kangasniemi2025single}, which in turn inherits many of its steps from the original proofs for quasiregular maps and mappings of finite distortion. In particular, if $f$ is not identically $y_0$, we first show that $f^{-1}\{y_0\}$ is totally disconnected, by showing that the Hausdorff dimension of $f^{-1}\{y_0\}$ is less than 1. We then show a local positivity result for the topological degree of $f$, which in turn is used to derive all three parts of Theorem \ref{open and discrete thm}.

However, there are two parts in the proof where our approach differs significantly from that of \cite{kangasniemi2025single}. The first such difference is centered around the Lusin (N) -property; recall that if $\Omega\subset\mathbb{R}^n$ is a domain and $f:\Omega\to\mathbb{R}^n$ is continuous, then $f$ satisfies the {\em Lusin (N) -property} if, for every $m_n$-nullset $E\subset\Omega$, the image $f E$ is an $m_n$-nullset. The Lusin (N) -property holds for quasiregular mappings by a result of Reshetnyak, see e.g.\ \cite{reshetni_ak1989space}, and also for $W^{1,n}_\loc$-regular mappings of finite distortion by a result of Kauhanen, Koskela, and Mal\'y \cite{kauhanen2001mappings}.

Notably, the proof of Theorem \ref{open and discrete thm} relies on $f$ satisfying the Lusin (N) -property. When $K$ is constant, this follows from a higher integrability and/or H\"older regularity result for $f$. However, in the case with variable $K$, there are no available higher integrability results, and the modulus of continuity shown in \cite{kangasniemi2025continuity} does not directly imply the Lusin (N) -property for $W^{1,n}$-maps; see the counterexamples of Koskela, Mal\'y, and Z\"urcher \cite{koskela2015luzin}, as well as the related positive result of Zapadinskaya \cite{zapadinskaya2014holder}.

For this reason, the Lusin (N) -property for solutions of \eqref{K,Sigma-quasiregular} and \eqref{K,Sigma-distortion} ends up being a result of independent interest. We record this result here.

\begin{thm}\label{thm:lusin_N}
    Let $\Omega\subset\mathbb{R}^n$ be a domain, and let $p,q\in[1,\infty]$.
    Suppose that
    $f\in W^{1,n}_{\loc}(\Omega,\mathbb{R}^n)$ satisfies the $(K,\Sigma)$-distortion inequality with defect or has a value of finite distortion at a point $y_0 \in \R^n$ with data $(K, \Sigma)$, where $K:\Omega\to[0,\infty)$ and $\Sigma:\Omega\to[0,\infty)$ are measurable functions such that
    \begin{equation*}
        K\in L^p_{\rm loc}(\Omega) \quad{\rm and}\quad
        \frac{\Sigma}{K}\in L^{q}_{\rm loc}(\Omega).
    \end{equation*}
    If $p^{-1}+q^{-1}<1$, then $f$ satisfies the Lusin (N) -property.
\end{thm}

The proof of Theorem \ref{thm:lusin_N} uses the key estimate \cite[Theorem 1.6]{kangasniemi2025continuity} behind the continuity of $W^{1,n}$-solutions to \eqref{K,Sigma-distortion}, along with ideas from the proof of the Lusin (N) -property for continuous pseudomonotone $W^{1,n}$-maps by Mal\'y and Martio \cite[Theorem A]{maly1995lusin}. We highlight that Theorem \ref{thm:lusin_N} in fact uses \cite[Theorem 1.6]{kangasniemi2025continuity} in a sharper manner than the main continuity result of its original paper. In particular, the continuity of $W^{1,n}$-solutions to \eqref{K,Sigma-distortion} arises from them being almost weakly monotone, see \cite[Definition 1.4]{kangasniemi2025continuity}, but the proof of Theorem \ref{thm:lusin_N} does not work for arbitrary almost weakly monotone $W^{1,n}$-maps.

The other major difference between the proofs of Theorem \ref{open and discrete thm} and \cite[Theorem 1.2]{kangasniemi2025single} occurs when we deduce that the topological degree of $f$ cannot vanish locally. This follows by showing that if it does vanish, then one obtains a continuity result for $\log \lvert f - y_0 \rvert$ which leads to a contradiction. The original version \cite[Lemma 5.5]{kangasniemi2025single} of this step uses technical estimates from \cite[Section 6]{kangasniemi2022heterogeneous}, but this approach is critically reliant on $K$ being constant.

Instead, we obtain this continuity result by again using the recent new techniques from \cite{kangasniemi2025continuity}. Notably, we cannot directly use the main result of \cite{kangasniemi2025continuity}, since when $K$ is non-constant, it is not clear to us whether $\log \lvert f - y_0 \rvert$ is in a $W^{1,n}$-space in this case. However, the underlying ideas still apply in our setting, as we observe in Proposition \ref{prop:Linfty_generalization}.

For Theorem \ref{thm:liouville}, the main part of its proof is to show that, given a mapping $f \in W^{1,n}_\loc(\R^n, \R^n)$ satisfying its assumptions, the integral of $J_f$ over the sublevel set $\{x \in \R^n : \abs{f(x) - y_0} < r\}$ vanishes for every $r > 0$. A critical part of this argument is a result which derives that the integral of $J_f$ over $\R^n$ vanishes using a combination of $f \in L^\infty(\R^n)$, $J_f^{-} \in L^1(\R^n)$, $\abs{Df}^n/K \in L^1(\R^n)$, and $\mathcal{K}_{n-1} < \infty$, without using any type of distortion inequality at all in the proof; see Lemma \ref{lem: J_f=0}. In some sense, this result acts as our counterpart to results such as \cite[Lemma 2.4]{kangasniemi2022heterogeneous}.

In \cite{kangasniemi2025single}, the proof of the Reshetnyak-type theorem for quasiregular values uses tools and ideas from the proof of the analogous Liouville-type theorem. Here, the final steps of our proof take the reverse approach, proving the Liouville-type Theorem \ref{thm:liouville} with the help of the Reshetnyak-type Theorem \ref{open and discrete thm}. This highlights how the proofs of these two types of results are closely intertwined, and that the results rely on a lot of similar ideas.

\section{Preliminaries}
In this section, we discuss several preliminary results that we use in our proofs. 

\subsection{The spherical logarithm}
We begin by recalling the \emph{spherical logarithm}, which has been a standard tool in the study of quasiregular values since \cite{kangasniemi2022heterogeneous}. Let $\Omega\subset\mathbb{R}^n$ be open, let $f\in W^{1,n}_{\rm loc}(\Omega,\mathbb{R}^n)$, and consider the map
\begin{equation}\label{spherical-log- formula}
    g:\mathbb{R}^n\to\mathbb{R}\times\mathbb{S}^{n-1},\quad g(x)=\left(\log|f(x)-y_0|,\frac{f(x)-y_0}{|f(x)-y_0|}\right).
\end{equation}
We denote the component function of $g$ by $g_\mathbb{R}:\Omega\to\mathbb{R}$ and $g_{\mathbb{S}^{n-1}}:\Omega\to\mathbb{S}^{n-1}$. We say that $g\in W^{1,p}_{\rm loc}(\Omega,\mathbb{R}\times\mathbb{S}^{n-1})$ if $g_\mathbb{R}\in W^{1,p}_{\rm loc}(\Omega)$ and $g_{\mathbb{S}^{n-1}}\in W^{1,p}_{\rm loc}(\Omega,\mathbb{R}^n)$, using the standard embedding $\mathbb{S}^{n-1}\subset\mathbb{R}^n$. We can then define a Jacobian determinant of $g$ a.e. on $\Omega$ by
\[
  J_g\,\vol_n
  = dg_{\mathbb{R}}\wedge g_{\mathbb{S}^{n-1}}^{\ast}\vol_{\mathbb{S}^{n-1}}
  = dg_{\mathbb{R}}\wedge g_{\mathbb{S}^{n-1}}^{\ast}
    \star d\!\left(2^{-1} |x|^2\right),
\]
where we use the fact that
$\star d\!\left(2^{-1} |x|^2\right)\in C^{\infty}(\wedge^{n-1}T^{\ast}\R^n)$
restricts to our chosen volume form $\vol_{\mathbb{S}^{n-1}}$ on $\mathbb{S}^{n-1}$.

We then proceed to study the map $g$ in our setting. For this, we require the following slight generalization of \cite[Lemma 5.4]{kangasniemi2022heterogeneous}; the proof is largely similar, though we recall it for the convenience of the reader.

\begin{lem}\label{lem: log_is_Sobolev}
    Let $p \in [1, \infty)$, let $B\subset\R^n$ be a ball, and let $y_0\in\R^n$. If $f\in W^{1,p}(B,\R^n)$ satisfies $|Df|/|f-y_0|\in L^p(B)$, then $f \equiv y_0$ almost everywhere, or $f^{-1}\{y_0\}$ is a null-set and $\log |f-y_0| \in W^{1,p}_\loc(B)$. 
\end{lem}
\begin{proof}
    Suppose that $f$ is not a.e.\ $y_0$, in which case there exists a $t_0 \in (1, \infty)$ and a set $F \subset B$ of positive measure such that $t_0^{-1} < |f-y_0| < t_0$ in $F$. We denote $h = \log |f-y_0|$, and for every $t \in (0, \infty)$, we denote $h_t = \log \max \{t, |f-y_0|\}$. By a Lipschitz-Sobolev chain rule, we obtain for every $t$ that $h_t \in W^{1,p}(B)$ and
    \begin{multline*}
        |\nabla h_t| 
        = \frac{|\nabla |f-y_0||}{|f-y_0|}
            \mathcal{X}_{\{ x \in B : |f(x)-y_0| > t\}}\\
        \le \frac{|Df|}{|f-y_0|} 
            \mathcal{X}_{\{ x \in B : |f(x)-y_0| > t\}}
        \le \frac{|Df|}{|f-y_0|}
    \end{multline*}
    a.e.\ in $B$; here, $\mathcal{X}_A$ denotes the characteristic function of a set $A \subset \R^n$.

    We first prove that $h \in L^1(B)$. Suppose towards contradiction that this is not the case: then, since the positive part $h^{+}$ of $h$ satisfies $h^{+} \le |f-y_0| \in L^1(B)$, we must have $h_B = -\infty$. Since $h_t$ converges pointwise to $h$ in a monotonically decreasing manner as $t \to 0$, the monotone convergence theorem yields that $\lim_{t \to 0} (h_t)_B = -\infty$. But now, by using the standard Sobolev-Poincar\'e inequality, we have for all $t \in (0, t_0^{-1})$ that
    \begin{multline*}
        m_n(F) \left( |(h_t)_B| - \log t_0\right)
        \le \int_F \left( |(h_t)_B| - |h|\right)\\ 
        = \int_F \left( |(h_t)_B| - |h_t|\right)
        \le \int_B |h_t - (h_t)_B|
        \le C(n,p,B) \lVert \nabla h_t \rVert_{L^p(B)}\\
        \le C(n,p,B) \left( \int_B \frac{|Df|^p}{|f-y_0|^p} 
            \right)^\frac{1}{p}.
    \end{multline*}
    This is a contradiction, since the lower bound tends to $\infty$ as $t \to 0$, but the upper bound is finite. 
    
    Thus, $h \in L^1(B)$. Since $h(x) = -\infty$ for all $x \in f^{-1}\{y_0\}$, we thus have that $f^{-1}\{y_0\}$ is a null-set. We also have
    \begin{align*}
        \lVert h_t-h\rVert_{L^1(B)}
        &= \int_{\{x \in B : |f(x)-y_0| \le t\}} |h - \log t|
        \xrightarrow{t \to 0} 0,
        \qquad \text{and}\\
        \left\lVert \nabla h_t  - \frac{\nabla |f-y_0|}{|f-y_0|}\right\rVert_{L^1(B)}
        &\le \int_{\{x \in B : |f(x)-y_0| \le t\}} \left| \frac{|Df|}{|f-y_0|} \right|
        \xrightarrow{t \to 0} 0,
    \end{align*}
    since $h \in L^1(B)$ and $|Df|/|f - y_0| \in L^p(B) \subset L^1(B)$. We conclude that $h \in W^{1,1}(B)$ with weak gradient $\nabla h = |f-y_0|^{-1} \nabla |f-y_0| \in L^p(B)$, which implies that $h \in W^{1,p}_\loc(B)$.
\end{proof}

Using Lemma \ref{lem: log_is_Sobolev}, we obtain the following result.

\begin{lem}\label{lem: spherical-log}
    Suppose that $\Omega\subset\R^n$ is open, $y_0\in\R^n$, and $f\in W^{1,n}_{\loc}(\Omega,\R^n)$ with
    $|Df|/|f-y_0|\in L^p_{\loc}(\Omega)$ for some $p\in (1,n)$. If $g:\Omega\to\R\times \mathbb{S}^{n-1}$ is as in \eqref{spherical-log- formula}, then
    $g\in W^{1,p}_{\loc}(\Omega,\R\times \mathbb{S}^{n-1})$ and
    \begin{equation}\label{eq:DG-and-JG}
        |Dg| = \frac{|Df|}{|f-y_0|},
        \qquad \text{and} \qquad
        J_g = \frac{J_f}{|f-y_0|^n}
    \end{equation}
    a.e.\ in $\Omega$. In particular, if $f$ has a value of finite distortion at $y_0$ with data $(K,\Sigma)$ for some measurable $K:\Omega\to[1,\infty)$ and $\Sigma:\Omega\to [0,\infty)$, then
    \begin{equation}\label{eq: spherical-log-distortion}
        |Dg|^n\le KJ_g+\Sigma
    \end{equation}
\end{lem}

\begin{proof}
    By Lemma \ref{lem: log_is_Sobolev}, $f^{-1}\{y_0\}$ is a null-set and $g_\mathbb{R}\in W^{1,p}_\loc(\Omega)$. In order to show the local Sobolev regularity of $g_{\mathbb{S}^{n-1}}$, let $U\subset\Omega$ be open
    and compactly contained in $\Omega$. For every $t>0$, we consider the function
    \begin{equation*}
        g_{\mathbb{S}^{n-1},t}:\Omega\to\mathbb{R}^n,
        \quad g_{\mathbb{S}^{n-1},t}=\frac{f-y_0}{\max\{|f-y_0|,t\}}.
    \end{equation*}
    By a Lipschitz-Sobolev chain rule  and the product rule of Sobolev spaces, $g_{\mathbb{S}^{n-1},t}$ is weakly differentiable. Moreover, we have $Dg_{\mathbb{S}^{n-1},t}=Dg_{\mathbb{S}^{n-1}}$ a.e. in $\{|f-y_0|\ge t\}$, and have $|Dg_{\mathbb{S}^{n-1},t}|=t^{-1}|Df|<|Df|/|f-y_0| $ a.e. in $\{|f-y_0|<t\}$; see e.g. \cite[Theorem 1.20]{heinonen2018nonlinear}. Thus, 
    $Dg_{\mathbb{S}^{n-1},t}\in L^p(U)$, and consequently $g_{\mathbb{S}^{n-1},t}\in W^{1,p}(U,\mathbb{R}^n)$. Moreover, if $t'<t$, we have
    \begin{equation*}
        \|g_{\mathbb{S}^{n-1},t}-g_{\mathbb{S}^{n-1},t'}\|^p_{W^{1,p}(U)}\le 2\int_{U\cap\{|f-y_0|<t\}}\left(1+\frac{|Df|^p}{|f-y_0|^p}\right)\xrightarrow{t\to 0}0,
    \end{equation*}
    which shows that the functions $g_{\mathbb{S}^{n-1},t}$ form Cauchy sequences in $W^{1,p}(U,\mathbb{R}^n)$
    as $t\to 0$. Since also $g_{\mathbb{S}^{n-1},t}\to g_{\mathbb{S}^{n-1}}$ pointwise as $t\to0$, we obtain that
    $g_{\mathbb{S}^{n-1}}\in W^{1,p}(U,\mathbb{R}^n)$ by completeness of Sobolev spaces.

    Thus, we obtain that $g\in W^{1,p}_{\loc}(\Omega,\mathbb{R}\times\mathbb{S}^{n-1})$. Observe that if $\varphi:\mathbb{R}\times\mathbb{S}^{n-1}$
    is the smooth conformal map $\varphi(t,x)=e^tx$, then $f=\varphi\circ g$ and $Df=D\varphi \circ Dg$ a.e. in $U$, and $|D\varphi(t,x)|^n=J_\varphi(t,x)=e^{nt}$ for all $(t,x)\in\mathbb{R}\times\mathbb{S}^{n-1}$. Thus,
    \begin{equation*}
        |Dg|=\frac{|Df|}{|D\varphi|\circ g}=\frac{|Df|}{|f-y_0|}\quad{\rm and}\quad J_g=\frac{J_f}{J_\varphi\circ g}=\frac{J_f}{|f-y_0|^n}
    \end{equation*}
    for a.e. $x\in \Omega$, which completes the proof.
\end{proof}

\subsection{Almost weakly monotone functions}
We recall that a continuous function $\varphi \colon \Omega \to \R$ is called \emph{monotone} if, for every open ball $B$ that is compactly contained in $\Omega$, the oscillation of $\varphi$ is attained on the boundary:
\begin{equation*}
    \mathop{\rm osc}_B\varphi=\mathop{\rm osc}_{\partial B}\varphi.
\end{equation*}
When applying similar ideas to very weak solutions to differential inequalities, such as mappings of finite distortion satisfying \eqref{K,Sigma-quasiregular} with $\Sigma\equiv0$, one needs to adopt the notion of so-called weakly monotone functions, due to Manfredi \cite{manfredi1994weakly}. A Sobolev function $\varphi:\Omega\to\mathbb{R}$ is called {\em weakly monotone} if, for every open ball $B$ that is compactly contained in $\Omega$ and all $M,m\in\mathbb{R}$ such that 
\begin{equation*}
    (\varphi|_B-M)^+,(m-\varphi|_B)^+\in W^{1,1}_0(B),
\end{equation*}
we have $(\varphi|_B-M)^+\equiv(m-\varphi|_B)^+\equiv0$.

If $\Sigma\not\equiv0$, the coordinate functions of solutions to \eqref{K,Sigma-quasiregular} need not be weakly monotone. Motivated by this, Onninen and Kangasniemi introduced a further quantitative weakening of monotonicity; see \cite[Definition 1.4]{kangasniemi2025continuity}. It is stated as follows:
\begin{defn}
    Let $\Omega\subset\mathbb{R}^n$ be open and let $\varphi\in W^{1,1}_{\rm loc}(\Omega)$. We say that $\varphi$ is $\alpha$-{\em almost weakly monotone} if there exist constant $C,\alpha>0$ such that, whenever $B_r\subset\Omega$ is a ball of radius $r$ that is compactly contained in $\Omega$ and $M,m\in\mathbb{R}$ are such that
    \begin{equation*}
        (\varphi|_{B_r}-M)^+\in W^{1,1}_0(B_r)\quad{\rm and}\quad(m-\varphi|_{B_r})^+\in W^{1,1}_0(B_r),
    \end{equation*}
    then we have
    \begin{equation*}
        \|(\varphi|_{B_r}-M)^+\|_{L^\infty(B_r)}\le Cr^\alpha\quad{\rm and}\quad  \|(m-\varphi|_{B_r})^+\|_{L^\infty(B_r)}\le Cr^\alpha.
    \end{equation*}
\end{defn}

We remark that if $\varphi\in W^{1,p}_{\rm loc}(\Omega)$, then all instances of $W^{1,1}_0$-spaces in the definition can be replaced with $W^{1,p}_0$-spaces.

\begin{lem}[{\cite[Lemma 4.1]{kangasniemi2025continuity}}]\label{W1p weak monotone lemma}
    Let $n\ge 2$, let  $\Omega\subset\mathbb{R}^n$ be open, let $p\in[1,\infty)$, and let $\varphi\in W^{1,p}_{\rm loc}(\Omega)$. Then $\varphi$ is $\alpha$-almost weakly monotone with constants $C,\alpha>0$ if and only, whenever $B_r\subset\Omega$ is a ball of radius $r$ that is compactly contained in $\Omega$ and $M,m\in\mathbb{R}$ are such that
    \begin{equation*}
        (\varphi|_{B_r}-M)^+\in W^{1,p}_0(B_r)\quad{\rm and}\quad(m-\varphi|_{B_r})^+\in W^{1,p}_0(B_r),
    \end{equation*}
    then we have
    \begin{equation*}
        \|(\varphi|_{B_r}-M)^+\|_{L^\infty(B_r)}\le Cr^\alpha\quad{\rm and}\quad  \|(m-\varphi|_{B_r})^+\|_{L^\infty(B_r)}\le Cr^\alpha.
    \end{equation*}
\end{lem}

It was shown in \cite{kangasniemi2025continuity} that, under the assumptions of Theorem \ref{open and discrete thm}, the coordinate functions of solutions to \eqref{K,Sigma-quasiregular} are $\alpha$-almost weakly monotone. The techniques outlined therein in fact yield the following estimate for the oscillation of $f_i$,
which plays an important role later in proving that $f$ satisfies the Lusin (N) -property.

\begin{lem}\label{osc lemma}
    Let $\Omega\subset\mathbb{R}^n$ be a domain, let $y_0\in\mathbb{R}^n$, and let $p,q\in[1,\infty]$ with $p^{-1}+q^{-1}<1$.
    Suppose that
    $f\in W^{1,n}_{\rm loc}(\Omega,\mathbb{R}^n)$ satisfies the $(K, \Sigma)$-distortion inequality with defect, where $K:\Omega\to[0,\infty)$ and $\Sigma:\Omega\to[0,\infty)$ are measurable functions such that
    \begin{equation*}
        K\in L^p_{\rm loc}(\Omega) \quad{\rm and}\quad
        \frac{\Sigma}{K}\in L^{q}_{\rm loc}(\Omega).
    \end{equation*}
    For all $i\in\{1,...,n\}$, let $f_i$ be the coordinate function of $f$. Then $f_i$ are $\alpha$-almost weakly monotone with constant $C=(n,p,q,K,\Sigma)$ and $\alpha= 1-p^{-1}-q^{-1}$. Moreover, for each $x$ and a.e. $r \in (0, \dist(x, \R^n \setminus \Omega))$, we have
    \begin{multline}\label{osc_estimate_DIWD}
        \left[\mathop{\rm ess\ osc}_{B(x,r)} f\right]^n \le C(n)r\int_{S(x,r)}|Df|^nd\mathcal{H}^{n-1}\\
        +C(p,q,n)\|K\|_{L^p(B(x,r))}\left\|\frac{\Sigma}{K}\right\|_{L^{q}(B(x,r))}r^{n(1-p^{-1}-q^{-1})}
    \end{multline}
\end{lem}

\begin{proof}
    Fix a point $x_0\in\Omega$ and a radius $R>0$ such that $\mathbb{B}^n(x_0,R)$ is compactly contained in $\Omega$. For $r\in(0,R)$, denote $B(r)=\mathbb{B}^n(x_0,r)$, $S(r)=\mathbb{S}^{n-1}(x_0,r)$. It follows by parts (1) and (3) of \cite[Lemma 2.2]{kangasniemi2025continuity} that there exists a Sobolev representative $\widetilde{f}$ of $f$ such that, for a.e. $r\in (0,R)$, 
    \begin{equation*}
        M(r):=\sup_{x\in S(r)} \widetilde{f}_i\quad{\rm and}\quad m(r):=\inf_{x\in S(r)} \widetilde{f}_i
    \end{equation*}
    belong to $\widetilde{f}_i(S(r))$, are finite, and we have
    \begin{equation*}
        (f_i|_{B(r)}-M(r))^+\in W^{1,n}_0(B(r))\quad{\rm and}\quad(m-f_i|_{B(r)})^+\in W^{1,n}_0(B(r)).
    \end{equation*}
    
    Let $\varphi:=(f_i-M(r))^+$ or $\varphi:=(m(r)-f_i)^+$ for a fixed $i \in \{1, \dots, n\}$. We then use \cite[Theorem 1.6]{kangasniemi2025continuity}, obtaining that
    \begin{equation}\label{L-infty estimate}
    \| \varphi\|^n_{L^\infty(B(r))}\le C(p,q,n)\|K\|_{L^p(B(r))}\left\|\frac{\Sigma}{K}\right\|_{L^{q}(B(r))}r^{n(1-p^{-1}-q^{-1})}.
    \end{equation}
    It follows that the coordinate function $f_i$ is $\alpha$-almost weakly monotone with constant $C=C(n,p,q,K,\Sigma)$ and $\alpha= 1-p^{-1}-q^{-1}$. 
    
    Moreover, using the fact that
    \begin{align*}
        \mathop{\rm ess\ sup}_{x\in B(r)}f_i(x)
        &\le M(r)+\|\varphi\|_{L^\infty(B(r))}
        \quad{\rm and}\quad\\ 
        \mathop{\rm ess\ inf}_{x\in B(r)}f_i(x)
        &\ge m(r)-\|\varphi\|_{L^\infty(B(r))},
    \end{align*}
    and combining this with \eqref{L-infty estimate} and part (2) of \cite[Lemma 2.2]{kangasniemi2025continuity}, it follows that
    \begin{align}\label{osc fi}
        \left[\mathop{\rm ess\ osc}_{B(r)} f_i\right]^n &=\left[\mathop{\rm ess\ sup}_{x\in B(r)}f_i(x)-\mathop{\rm ess\ inf}_{x\in B(r)}f_i(x)\right]^n
        \\
        \nonumber&\le 2^{n-1}(M(r)-m(r))^n+2^n\|\varphi\|^n_{L^\infty(B(r))}\\
        \nonumber&\le C(n)r\int_{S(r)}|\nabla f_i|^n d\mathcal{H}^{n-1}\\
        \nonumber&\qquad+C(p,q,n)\|K\|_{L^p(B(r))}\left\|\frac{\Sigma}{K}\right\|_{L^{q}(B(r))}r^{n(1-p^{-1}-q^{-1})}.
    \end{align}
    Finally, by summing over $i=1,\dots,n$ on both sides of \eqref{osc fi}, we obtain
    \begin{align*}
            \left[\mathop{\rm ess\ osc}_{B(r)} f\right]^n
            &\le\sum_{i=1}^n \left[\mathop{\rm ess\ osc}_{B(r)} f_i\right]^n\\
            &\le C(n)r\int_{S(r)}|Df|^nd\mathcal{H}^{n-1}\\
            &\hspace{1.5cm}+C(p,q,n)\|K\|_{L^p(B(r))}\left\|\frac{\Sigma}{K}\right\|_{L^{q}(B(r))}r^{n(1-p^{-1}-q^{-1})}.
    \end{align*}
\end{proof}

\subsection{Degree theory and local index}\label{sec: degree and index}
We then recall the definition of the topological degree. For our purposes, it suffices that this definition is valid for sufficiently regular maps. 

\begin{defn}
    Let $\Omega\subset\mathbb{R}^n$ be a domain, let $U$ be an open set compactly contained
    in $\Omega$, and let $f\in C(\Omega,\mathbb{R}^n)\cap W^{1,n}_{\rm loc}(\Omega,\mathbb{R}^n)$. Suppose that $f$ satisfies the Lusin (N)-property. 
    The {\em topological degree} deg $(f,.,U)$ is the unique integer-valued locally constant function on $\mathbb{R}^n\backslash f\partial U$  which satisfies
    \begin{equation}\label{degree formula}
        \int_U(v\circ f(x))J_f(x)dm_n(x)=\int_{\mathbb{R}^n\backslash f\partial U}v(y)\ {\rm deg}(f,y,U)dm_n(y).
    \end{equation}
    for all $v\in L^\infty(\mathbb{R}^n)$. Note that $v\circ f$ may not be measurable, but the quantity $(v\circ f)J_f$ is measurable under our assumptions by e.g. \cite[Proposition 5.22]{fonseca1995degree}.
\end{defn}

We note that the aforementioned topological degree is a special case of a more general construction of the classical topological degree for continuous mappings $f\in C(\Omega,\mathbb{R}^n)$. The equivalence of these concepts follows from ideas similar to \cite[Theorem 5.27]{fonseca1995degree}. We also recall the following standard properties of the topological degree that we use; for proofs, we refer to e.g.\ \cite[Theorems 2.1, 2.3 (2), 2.7]{fonseca1995degree}

\begin{lem}\label{basic degree property}
    Let $\Omega\subset\mathbb{R}^n$ be a domain, let $f\in C(\Omega,\mathbb{R}^n)\cap W^{1,n}_{\rm loc}(\Omega,\mathbb{R}^n)$ satisfy
    the Lusin (N) -property, let $U\subset\Omega$ be open and compactly contained in $\Omega$, and let $y\notin f\partial U$. The topological degree $\deg(f,y,U)$ satisfies the following properties. 
    \begin{enumerate}[noitemsep]
        \item \label{enum:degree_vanishing} (Vanishing outside the image set) If $\deg(f,y,U)=0$, then $y\in fU$.
        \item (Additivity) If $U=\bigcup_{i\in I}U_i$, where $I$ is countable, $U_i$ are mutually disjoint open sets, and $y\notin f\partial U_i$ for every $i\in I$, then
        \begin{equation*}
            \deg(f,y,U)=\sum_{i\in I} \deg(f,y,U_i).
        \end{equation*}
        \item (Excision) If $K\subset U$ is compact and $y\notin fK$, then
        \begin{equation*}
            \deg(f,y,U)=\deg(f,y,U\setminus K).    
        \end{equation*}
        \item (Homotopy invariance) If $g\in C(\Omega,\mathbb{R}^n)\cap W^{1,n}_{\rm loc}(\Omega,\mathbb{R}^n)$ is another map satisfying the Lusin (N) -property, and if $H: U\times I\to\mathbb{R}^n$ is a homotopy from $f$ to $g$ for which $y\notin H((\partial U)\times I)$, then
        \begin{equation*}
            \deg(f,y,U)=\deg(g,y,U).    
        \end{equation*}
    \end{enumerate}
\end{lem}

We end this section by recalling how the topological degree leads to a definition of the local index.

\begin{defn}
    Let $\Omega\subset\mathbb{R}^n$ be a domain, $f\in C(\Omega,\mathbb{R}^n)\cap W^{1,n}_{\rm loc}(\Omega,\mathbb{R}^n)$ satisfies the Lusin (N) -property. Suppose that $x \in \Omega$ is an isolated point of $f^{-1}\{f(x)\}$. Then there exists an open neighborhood $U$ of $x$ such that $U$ is compactly contained in $\Omega$ and $\overline{U}\cap f^{-1}\{f(x)\}=\{x\}$. The excision property of Lemma \ref{basic degree property} yields that if $U, V \subset \Omega$ are open, compactly contained in $\Omega$, and satisfy $U \cap f^{-1}\{f(x)\} = \{x\} = V \cap f^{-1}\{f(x)\}$, then $\deg(f, f(x), U) = \deg(f, f(x), V)$. This unique value of $\deg(f, f(x), U)$ for $U \subset \Omega$ with $U \cap f^{-1}\{f(x)\} = \{x\}$ is called the \emph{local index $i(x, f)$ of $f$ at $x$}; see e.g.\ \cite[Definition 2.8]{fonseca1995degree} for details.
\end{defn}

\section{A refinement of an $L^\infty$-estimate}

In our argument later on in this article, we use a variant of the estimate \cite[Theorem 1.6]{kangasniemi2025continuity}. Although this result is already originally stated in a rather general form, this form is regardless insufficient for us, since it is stated for functions in a $W^{1,n}_0$-space.

We have already observed multiple other potential uses for this slight generalization of \cite[Theorem 1.6]{kangasniemi2025continuity}. Thus, in anticipation of these other uses, we state an abstract, technical version of this generalization not tied to our specific application. The proof of this version largely follows the same steps as the proof of the original result, with a few minor changes.

\begin{prop}\label{prop:Linfty_generalization}
	Let $n \ge 2$, let $\Omega \subset \R^n$ be open and bounded, let $p,q \in (1, \infty]$ with $p^{-1} + q^{-1} < 1$, and let $\varphi \in W^{1,1}_0(\Omega)$ with $\varphi \ge 0$. Suppose that there exist functions $K \in L^p(\Omega)$, $A \in L^q(\Omega)$, and $J \in L^1(\Omega)$ such that $K > 0$, $A \ge 0$, 
	\begin{align}
		\abs{\nabla \varphi(x)}^n &\le K(x)(J(x) + A(x)) \label{eq:Linfty_generalization_dist}
		&&\text{for a.e.\ }x \in \Omega, \text{ and}\\
		&\int_{\varphi^{-1}(t, \infty)} J = 0 \label{eq:Linfty_generalization_integral}
		&&\text{for all }t \in [0, \infty).
	\end{align}
	Then
	\[
		\norm{\varphi}_{L^\infty(\Omega)}^n \le C(n,p,q) 
			\norm{K}_{L^p(\Omega)} \norm{A}_{L^q(\Omega)} [m_n(\Omega)]^{1-\frac{1}{p}-\frac{1}{q}}.
	\]
\end{prop}

In particular, beyond the lowered Sobolev exponent of $\varphi$, the main difference here is that in \cite[Theorem 1.6]{kangasniemi2025continuity}, $J$ is the inner product of $\nabla \varphi$ with a divergence-free vector field, whereas in Proposition \ref{prop:Linfty_generalization}, $J$ can be otherwise unrelated to $\nabla \varphi$ as long as it satisfies \eqref{eq:Linfty_generalization_dist} and \eqref{eq:Linfty_generalization_integral}.

The first step in the proof is a version of \cite[Lemma 2.4]{kangasniemi2025continuity}. The proof is practically identical to the proof of the original version, though we recall the argument for the convenience of the reader. Here, for a function $J \colon \Omega \to \R$, we denote $J^{+} = \max(J, 0)$ and $J^{-} = \max(-J, 0)$.

\begin{lem}\label{lem:weighted_zero_integrals}
	Let $n \ge 2$, let $\Omega \subset \R^n$ be open and bounded, let $\varphi \colon \Omega \to [0, \infty]$ be measurable, and let $J \in L^1(\Omega)$ be such that \eqref{eq:Linfty_generalization_integral} holds. Let $F \colon [0, \infty] \to [0, \infty]$ be a non-decreasing left-continuous function such that $F(\varphi(x)) < \infty$ for a.e.\ $x \in \Omega$. Suppose that
	\[
		\int_{\varphi^{-1}(0, \infty)} F(\varphi) J^{+} < \infty 
		\quad \text{or} \quad
		\int_{\varphi^{-1}(0, \infty)} F(\varphi) J^{-}  < \infty. 
	\] 
	Then $F(\varphi) J \in L^1(\Omega)$ and
	\[
		\int_{\varphi^{-1}(0, \infty)} F(\varphi) J = 0.
	\]
\end{lem}
\begin{proof}
	The case where $F$ is constant is trivial by using \eqref{eq:Linfty_generalization_integral} with $t = 0$. Otherwise, fix $\eps > 0$. By using \cite[Lemma 2.1]{kangasniemi2025continuity}, we find a piecewise constant staircase approximation $G \colon [0, \infty] \to [0, \infty]$ of $F$; more precisely, we have an increasing (ordinal-indexed) sequence $0 = t_1 < t_2 < \dots < t_{\omega} < t_{\omega+1} = \infty$ such that the function 
	\[
		G \colon [0, \infty] \to [0, \infty],\quad
		\begin{cases}
			G(0) = F(0),\\
			G(t) = F(t_{i+1}), & \text{if } t \in (t_i, t_{i+1}], i \in \Z_{> 0}\\
			G(t_\omega) = F(t_\omega),\\
			G(t) = F(t) = F(\infty), & \text{if } t \in (t_\omega, t_{\omega+1}].
		\end{cases}
	\]
	satisfies $\abs{G(t) - F(t)} < \eps$ for all $t \in [0, \infty]$.
	
	Now, for all $i \in \{1, \dots, \omega\}$, by applying \eqref{eq:Linfty_generalization_integral} with $t = t_i$ and $t = t_{i+1}$, we observe that the integral of $J$ over $\varphi^{-1} (t_i, t_{i+1}]$ vanishes. Thus, for all such $i$, we have
	\begin{multline*}
		\int_{\varphi^{-1} (t_i, t_{i+1}]} G(\varphi) J^{+}
		= F(t_{i+1}) \int_{\varphi^{-1} (t_i, t_{i+1}]} J^{+}\\
		= F(t_{i+1}) \int_{\varphi^{-1} (t_i, t_{i+1}]} J^{-}
		= \int_{\varphi^{-1} (t_i, t_{i+1}]} G(\varphi) J^{-}.
	\end{multline*}
	Note that if $i = \omega$ and $F(t_{i+1}) = \infty$, then the result of the above computation remains valid; indeed, since $F(t) = F(\infty) = \infty$ for $t \in (t_\omega, t_{\omega + 1}]$, and since $F(\varphi(x)) < \infty$ for a.e.\ $x \in \Omega$, it follows that $\varphi^{-1} (t_\omega, t_{\omega + 1}]$ is a null-set. Moreover, by combining all the above identities with a use of \eqref{eq:Linfty_generalization_integral} for $t = 0$, we also get an analogous integral identity over $\varphi^{-1}\{t_\omega\}$.
	
	Thus, it follows that
	\[
		\int_{\varphi^{-1} (0, \infty)} G(\varphi) J^{+}
		= \int_{\varphi^{-1} (0, \infty)} G(\varphi) J^{-}
	\]
	But now, since $\abs{G(t) - F(t)} \le \eps$ for all $t \in [0, \infty]$, we get
	\[
		\abs{\int_{\varphi^{-1} (0, \infty)} 
			F(\varphi) J^{+} - \int_{\varphi^{-1} (0, \infty)} F(\varphi) J^{-}}
		\le \eps \int_\Omega \abs{J},
	\]
	where the difference of integrals is well defined as at least one term is finite. Since $J \in L^1(\Omega)$, the claim follows by letting $\eps \to 0$.
\end{proof}

The function $F$ this lemma will be applied on is the distribution function of $\varphi$. For this, suppose that $\Omega \subset \R^n$ is open and bounded, and $\varphi \colon \Omega \to [0, \infty]$ is measurable. The \emph{upper distribution function} $\mu^+_\varphi \colon [0, \infty] \to [0, \abs{\Omega}]$ is defined by
\[
	\mu^{+}_\varphi(t) = m_n(\{x \in \Omega : \varphi(x) \ge t\}) = m_n(\varphi^{-1} [t, \infty])
\]
The function $\mu^{+}_\varphi$ is hence a non-increasing left-continuous Borel function on $[0, \infty]$.

For the proof, we require two estimates for this function. The first, shown e.g.\ in \cite[Lemma 3.2]{kangasniemi2025continuity}, yields that for every $\gamma \in (0, 1)$, we have
\begin{equation}\label{eq:distr_upper_est}
	\int_\Omega \frac{1}{(\mu_\varphi^{+} \circ \varphi)^\gamma} 
	\le \frac{[m_n(\Omega)]^{1-\gamma}}{1 - \gamma}.
\end{equation}
This estimate is sensitive to us using the upper distribution function, as for the \emph{lower distribution function} $\mu^{-}_\varphi$ given by $\mu^{-}_\varphi(t) = m_n(\{x \in \Omega : \varphi(x) > t\})$, the direction of the estimate \eqref{eq:distr_upper_est} reverses. 

The other estimate we require is a Sobolev-type inequality using the distribution function given in \cite[Proposition 1.7]{kangasniemi2025continuity}. Namely, if we additionally have $\varphi \in W^{1,1}_0(\Omega)$, then
\begin{equation}\label{eq:superlevel_Sobolev}
	\norm{\varphi}_{L^\infty(\Omega)} \le C(n) \int_\Omega \frac{\abs{\nabla \varphi}}{(\mu_\varphi^{+} \circ \varphi)^\frac{n-1}{n}}.
\end{equation}
Here, $C(n)$ is in fact exactly the constant of the $n$-dimensional isoperimetric inequality. See also the recent work of Harjulehto and Hurri-Syrj\"anen \cite[Section 4]{Harjulehto-HurriSyrjanen} for a generalization using Hausdorff content.

We are now ready to prove Proposition \ref{prop:Linfty_generalization}. The proof is mostly identical with the proof of \cite[Theorem 1.6]{kangasniemi2025continuity}, with only a few minor differences.

\begin{proof}[Proof of Proposition \ref{prop:Linfty_generalization}]
	We fix a $\gamma \in (p^{-1}, 1 - q^{-1})$, noting that the interval is non-empty by our assumption, and denote $U = \varphi^{-1}(0, \infty)$. We first observe that since $\abs{\nabla \varphi}$ is non-negative and $K$ is positive, \eqref{eq:Linfty_generalization_dist} in fact implies that $J + A \ge 0$ a.e.\ in $\Omega$. From this, it follows that
	\[
		J^{-} \le A \qquad \text{a.e.\ in }\Omega.
	\]
	Thus, we have by H\"older's inequality that
	\[
		\int_{U} \frac{J^{-}}{(\mu_\varphi^{+} \circ \varphi)^\gamma}
		\le \int_{U} \frac{A}{(\mu_\varphi^{+} \circ \varphi)^\gamma}
		\le \norm{A}_{L^q(\Omega)} \left( \int_\Omega 
			\frac{1}{(\mu_\varphi^{+} \circ \varphi)^\frac{\gamma q}{q-1}} \right)^{1 - \frac{1}{q}}. 
	\]
	Since $0 < \gamma q /(q-1) < 1$ by our assumptions, we may thus apply \eqref{eq:distr_upper_est} to conclude that
	\begin{multline}\label{eq:Linfty_norm_estimate_step_1}
		\int_{U} \frac{J^{-}}{(\mu_\varphi^{+} \circ \varphi)^\gamma}
		\le \int_{U} \frac{A}{(\mu_\varphi^{+} \circ \varphi)^\gamma}\\
		\le C(n, q, \gamma) \norm{A}_{L^q(\Omega)} [m_n(\Omega)]^{1 - \frac{1}{q} - \gamma} < \infty.
	\end{multline}
	
	Following these initial steps, we begin the main branch of the computation. We first apply \eqref{eq:superlevel_Sobolev} to obtain that
	\begin{align*}
		\norm{\varphi}_{L^\infty(\Omega)} &\le C(n) \int_\Omega \frac{\abs{\nabla \varphi}}{(\mu_\varphi^{+} \circ \varphi)^\frac{n-1}{n}}\\
		&\le C(n) \int_\Omega \frac{\abs{\nabla \varphi}}{K^\frac{1}{n} (\mu_\varphi^{+} \circ \varphi)^\frac{\gamma}{n}}
		\cdot K^\frac{1}{n} \cdot \frac{1}{(\mu_\varphi^{+} \circ \varphi)^{\frac{n-1-\gamma}{n}}}.
	\end{align*}
	Since $n \ge 2$ and $\gamma < 1$, we have $n - 1 - \gamma > 0$. Now, we use H\"older's inequality with exponents $n$, $pn$, and $np/(np - p - 1)$, noting that $np - p - 1 \ge 2p - p - 1 > 0$ so the final exponent is indeed in $[1, \infty)$. We get that
	\[
		\norm{\varphi}_{L^\infty(\Omega)}^n\\
		\le C(n) \norm{K}_{L^p(\Omega)} \left( \int_\Omega \frac{1}{(\mu_\varphi^{+} \circ \varphi)^\frac{n - 1 - \gamma}{n - 1 - 1/p}} \right)^{n - 1 - \frac{1}{p}} 
		\int_\Omega \frac{\abs{\nabla \varphi}^n}{K (\mu_\varphi^{+} \circ \varphi)^\gamma}. 
	\]
	The exponent $(n - 1 - \gamma)/(n - 1 - 1/p)$ is again between $0$ and $1$, so we can use \eqref{eq:distr_upper_est}, obtaining that
	\begin{equation}\label{eq:Linfty_norm_estimate_step_2}
		\norm{\varphi}_{L^\infty(\Omega)}^n\\
		\le C(n,p,\gamma) \norm{K}_{L^p(\Omega)} \abs{\Omega}^{\gamma - \frac{1}{p}} \int_\Omega \frac{\abs{\nabla \varphi}^n}{K (\mu_\varphi^{+} \circ \varphi)^\gamma}.
	\end{equation}
	
	Now, we estimate the final term in \eqref{eq:Linfty_norm_estimate_step_2}. We first note that $\varphi$ is identically zero in $\Omega \setminus U$. Thus, $\nabla\varphi(x) = 0$ for a.e.\ $x \in \Omega \setminus U$. It follows using \eqref{eq:Linfty_generalization_dist} that
	\[
		\int_\Omega \frac{\abs{\nabla \varphi}^n}{K (\mu_\varphi^{+} \circ \varphi)^\gamma}
		= \int_U \frac{\abs{\nabla \varphi}^n}{K (\mu_\varphi^{+} \circ \varphi)^\gamma}
		\le \int_U \frac{J}{(\mu_\varphi^{+} \circ \varphi)^\gamma}
		+ \int_U \frac{A}{(\mu_\varphi^{+} \circ \varphi)^\gamma}.
	\]
	Now, due to \eqref{eq:Linfty_norm_estimate_step_1}, Lemma \ref{lem:weighted_zero_integrals} applies and yields that the $J$-term vanishes. We also use \eqref{eq:Linfty_norm_estimate_step_1} to estimate the $A$-term, obtaining that
	\begin{equation}\label{eq:Linfty_norm_estimate_step_3}
		\int_\Omega \frac{\abs{\nabla \varphi}^n}{K (\mu_\varphi^{+} \circ \varphi)^\gamma}
		\le \int_U \frac{A}{(\mu_\varphi^{+} \circ \varphi)^\gamma}
		\le C(n, q, \gamma) \norm{A}_{L^q(\Omega)} [m_n(\Omega)]^{1 - \frac{1}{q} - \gamma}.
	\end{equation}
	By chaining together \eqref{eq:Linfty_norm_estimate_step_2} and \eqref{eq:Linfty_norm_estimate_step_3}, it follows that
	\[
		\norm{\varphi}_{L^\infty(\Omega)}^n
		\le C(n,p,q,\gamma) \norm{K}_{L^p(\Omega)} \norm{A}_{L^q(\Omega)} 
		[m_n(\Omega)]^{1 - \frac{1}{q} - \frac{1}{p}}.
	\]
	As the choice of $\gamma$ depended only on $p$ and $q$, the claim follows.
\end{proof}

\section{The Lusin (N) -property for solutions of \eqref{K,Sigma-distortion}}
In this section, we prove Theorem \ref{thm:lusin_N}, which yields a condition for when a map satisfying \eqref{K,Sigma-distortion} has the Lusin (N) -property. Our approach is a combination of an argument by Mal\'y and Martio \cite[Theorem A]{maly1995lusin} that applies to mappings of finite distortion, and the oscillation estimate recalled in Lemma \ref{osc lemma}.

Recall that if $\Omega\subset\mathbb{R}^n$ is a domain and $f:\Omega\to\mathbb{R}^n$ is continuous, then $f$ satisfies the {\em Lusin (N) -property} if, for every $m_n$-nullset $E\subset\Omega$, the image $f E$ is an $m_n$-nullset. Our next objective is to prove that, under the assumptions of Theorem \ref{open and discrete thm}, any solution of the distortion inequality with defect \eqref{K,Sigma-distortion} satisfies the Lusin (N) -property. 

We first isolate a convenient lemma which is included in the proof of \cite[Theorem A]{maly1995lusin} by Mal\'y and Martio.

\begin{lem}\label{lem:Maly-Martio_lemma}
    Let $\Omega \subset \R^n$ be open, and let $f \in W^{1,n}_{\rm loc}(\Omega, \R^m)$ with $m \in \Z_{>0}$. Then for every $x \in \Omega$, there exists a family of radii $R(x) \subset (0, \dist(x, \R^n \setminus \Omega))$ such that $R(x) \cap (0, \delta)$ has positive measure for every $\delta > 0$, and we have
    \begin{equation}\label{eq:Maly-Martio_estimate}
        r \int_{S(x, r)} \lvert D f \rvert^n \, d\mathcal H^{n-1}
        \le C(n) \int_{B(x, r)} \left( \lvert D f \rvert^n + 1 \right)  
    \end{equation}
    for every $r \in R(x)$.
\end{lem}

While one can prove Lemma \ref{lem:Maly-Martio_lemma} by directly following the proof of \cite[Theorem A]{maly1995lusin}, we elect to instead provide a more streamlined proof. Indeed, Lemma \ref{lem:Maly-Martio_lemma} is a direct consequence of the following basic estimate for absolutely continuous functions.

\begin{lem}\label{lem: new proof}
    Let $\omega:[0,R)\to[0,\infty)$ be a non-negative absolutely continuous function, and let $n\in(0,\infty)$. Then there exists a family $S\subset[0,R)$ of radii such that $S\cap(0,\delta)$ has positive measure for all $\delta\in(0,R)$, and we have
    \begin{align*}
        r\omega'(r)\le n\omega(r)+r^n
    \end{align*}
    for all $r\in S$.
\end{lem}

\begin{proof}
    Suppose towards contradiction that there exists a $\delta\in(0,R)$ such that 
    \begin{align*}
        r \omega'(r)>n\omega(r)+r^n
    \end{align*}
    for a.e. $r\in(0,\delta)$. Note that by dividing on both sides by $r^{n+1}$, this yields
    \begin{align*}
        r^{-n}\omega'(r)-nr^{-n-1}\omega(r)>r^{-1}
    \end{align*}
    for a.e. $r\in(0,\delta)$. Then for every $r\in(0,\delta)$, we have
    \begin{align*}
        \omega(\delta)/\delta^n&\ge\omega(\delta)/\delta^n-\omega(r)/r^n=\int_r^\delta\frac{d}{d\rho}\frac{\omega(\rho)}{\rho^{n}}d\rho\\
        &=\int_r^\delta(\rho^{-n}\omega'(\rho)-n\rho^{-n-1}\omega(\rho))d\rho\ge\int_r^\delta\rho^{-1}d\rho=\log(\delta/r).
    \end{align*}
    By letting $r\to0$, we get $\omega(\delta)/\delta^n\ge\infty$, which contradicts the fact that $\delta>0$ and $\omega$ is finite-valued. Thus, the claim holds.
\end{proof}

\begin{proof}[Proof of Lemma \ref{lem:Maly-Martio_lemma}]
    Given $x \in \Omega$, the result is an immediate consequence of applying Lemma \ref{lem: new proof} to the absolutely continuous function $\omega \colon [0, \dist(x, \R^n \setminus \Omega)) \to [0, \infty)$ given by
    \[
        \omega(r) = \int_{B(x,r)} \lvert Df\rvert^n, \qquad \text{which satisfies} \qquad
        \omega'(r) = \int_{S(x,r)} \lvert Df\rvert^n \, d\mathcal{H}^{n-1}.
    \]
\end{proof}

Now, we are ready to prove Theorem \ref{thm:lusin_N}.

\begin{proof}[Proof of Theorem \ref{thm:lusin_N}]
    We begin with the case where $f$ satisfies the $(K, \Sigma)$-distortion inequality with defect. Let $E\subset\Omega$ with $|E|=0$. We first restrict ourselves to the case in which there exists an open set $U$ such that $E\subset U$  and $U$ is compactly contained in $\Omega$. 
    Let $V$ be an open set such that $E\subset V\subset U$, and for every $x \in U$, let $R(x)$ be the family of radii provided by Lemma \ref{lem:Maly-Martio_lemma}. Now, by Lemma \ref{osc lemma}, for every $x \in E$, we may select a radius $r_x\in R(x)$ such that $B(x,r_x)$ is compactly contained in $V$ and
    \begin{equation}\label{osc control}
        \begin{aligned}
        \left[\mathop{\rm ess\ osc}_{B(x,r_x)} f\right]^n
        &\le C(n)r_x\int_{S(x,r_x)}|Df|^nd\mathcal{H}^{n-1}\\
        &\;+C(p,q,n)\|K\|_{L^p(B(x,r_x))}\left\|\frac{\Sigma}{K}\right\|_{L^{q}(B(x,r_x))}r_x^{n(1-p^{-1}-q^{-1})}.
        \end{aligned}
    \end{equation}
    
    We now construct a covering of $fE$ as follows:
    \begin{equation*}
        \mathcal{F}:=\left\{B\Big(f(x),\mathop{\rm ess\ osc}_{B(x,r_x)} f\Big): x\in E \right\}.
    \end{equation*}
    By the Vitali covering Lemma, we find a subfamily 
    \[
        \mathcal F' = \left\{ B\Big(f(x_i), \mathop{\rm ess\ osc}_{B(x_i, r_{x_i})} f\Big) : i \in I \right\}  \subset  \mathcal F
    \]
    of pairwise disjoint balls such that $\{ 5B : B \in \mathcal F'\}$ covers $fE$. Now, by \eqref{eq:Maly-Martio_estimate}, \eqref{osc control}, and H\"older's inequality with exponents $p,q$ and $(1-p^{-1}-q^{-1})^{-1}$, we have
    \begin{equation*}
        \begin{aligned}
            |fE|&\le \left|\bigcup_{i \in I} B(f(x_i),5 \mathop{\rm ess\ osc}_{B(x_i,r_{x_i})} f)\right|\le C(n)\sum_{i \in I}\left[\mathop{\rm ess\ osc}_{B(x_i,r_{x_i})} f\right]^n\\
            &\le C(n)\sum_{i \in I} \int_{B(x_i,r_{x_i})}\left(1+|Df(y)|^n\right)dy\\
            &\quad +C(p,q,n)\sum_{i \in I}\|K\|_{L^p(B(x_i,r_{x_i}))}\left\|\frac{\Sigma}{K}\right\|_{L^{q'}(B(x_i,r_{x_i}))}r_{x_i}^{n(1-p^{-1}-q^{-1})}\\
            &\le C(n)\int_V\left(1+|Df(y)|^n\right)dy
            \\
            &\quad+C(p,q,n)\left(\int_VK^p\right)^{p^{-1}}
            \left(\int_V\left(\frac{\Sigma}{K}\right)^{q}\right)^{q^{-1}}
            \left(\sum_{i \in I} r_{x_i}^n\right)^{1-p^{-1} - q^{-1}}.
        \end{aligned}
    \end{equation*}
    Since $E\subset V\subset U\subset\Omega$, $f\in W^{1,n}(U)$, $K\in L^p(U)$, $\Sigma/K\in L^q(U)$ and $\sum_ir^n_{x_i}\le C(n)|V|$, letting $|V|\to 0$, we conclude that $|fE|=0$.

    We then consider the general case where $E\subset\Omega$ with $|E|=0$. We define a sequence of open sets $(U_i)_{i=1}^\infty$ that approximates $\Omega$ from the inside as follows:
    \begin{equation*}
        U_i:=\left\{x\in \Omega: {\rm dist}(x,\Omega^c)<\frac{1}{i}\right\}.
    \end{equation*}
    By the argument above, we have $|f(E\cap U_i)|=0$ for every $i \in \Z_{>0}$, and hence
    \begin{equation*}
        |fE|\le \sum_i|f(E\cap U_i)|=0.
    \end{equation*}

    Finally, we consider the case where $f$ instead has a value of finite distortion with data $(K, \Sigma)$ at some point $y_0 \in \R^n$. Then $f$ satisfies the $(K, \abs{f-y_0}^n \Sigma)$-distortion inequality with defect. Since $f \in W^{1,n}_\loc(\Omega, \R^n)$, we have by the Sobolev embedding theorem that $\abs{f-y_0} \in L^r_\loc(\Omega)$ for every $r \in [1, \infty)$. Thus, if we select a $q' \in (1, q)$ such that $1/p + 1/q' < 1$, then $\abs{f-y_0}^n \Sigma \in L^{q'}_\loc(\Omega)$. The claim then follows from the previously covered case for mappings satisfying a distortion inequality with defect.
\end{proof}

\section{Disconnectedness of preimages}
In this section, we show the first step of Theorem \ref{open and discrete thm}, namely that $f^{-1}\{y_0\}$ is disconnected. The method to prove this is to show that $\mathcal{H}^{-1}\{y_0\}=0$, by using a version of the argument in \cite{onninen2008mappings}. 

We begin by recalling a general Caccioppoli-type estimate; see \cite[Lemma 2.1]{onninen2008mappings}. For another version of this estimate with slightly more general assumptions; see \cite[Lemma 6.2]{kangasniemi2022heterogeneous}. We state the latter here.
\begin{lem}[{\cite[Lemma 6.2]{kangasniemi2022heterogeneous}}]\label{tech lemma} 
    Let $\Omega\subset\mathbb{R}^n$ be a domain, and let $f\in W^{1,n}_{\rm loc}(\Omega,\mathbb{R}^n)\cap L^\infty_{\rm loc}(\Omega,\mathbb{R}^n)$. If $\Psi:[0,\infty)\to\mathbb{R}$ is a piecewise $C^1$-smooth function with $\Psi'$ locally bounded, then for every test function $\eta\in C^\infty_0(\Omega)$, we have 
    \begin{equation*}
        \left|\int_\Omega\eta\left[n\Psi(|f|^2)+2|f|^2\Psi'(|f|^2)\right]J_f\right|\le\sqrt{n}\int_\Omega|\nabla\eta||f|\Psi(|f|^2)|Df|^{n-1}.
    \end{equation*}
\end{lem}

We now apply Lemma \ref{tech lemma} to obtain a Sobolev norm estimate for the function $\min(\log\log|f-y_0|^{-1},k)$, where $k\in\mathbb{Z}_{>0}$.

\begin{lem}\label{Duk bounded lemma}
    Let $\Omega\subset\mathbb{R}^n$ be a domain, let $y_0\in\mathbb{R}^n$, and let $p>n-1,q>1$ with $p^{-1}+q^{-1}<1$.
    Suppose that
    $f\in W^{1,n}_{\rm loc}(\Omega,\mathbb{R}^n)$ has a value of finite distortion at $y_0\in\mathbb{R}^n$ with data $(K, \Sigma)$, where $K:\Omega\to[1,\infty)$ and $\Sigma:\Omega\to[0,\infty)$ are measurable functions such that
    \begin{equation}
        K\in L^p_{\rm loc}(\Omega) \quad{\rm and}\quad
        \frac{\Sigma}{K}\in L^{q}_{\rm loc}(\Omega).
    \end{equation} 
    Let $u_k=\min(\log\log|f-y_0|^{-1},k)$, then there exists a neighborhood $V$ of $y_0$ such that
    \begin{equation}\label{Duk uniformly bounded}
    \int_D|\nabla u_k|^\beta\le C_D<\infty,
    \end{equation}
    for $\beta=\frac{pn}{p+1}<n$, all $k\in\mathbb{Z}_{>0}$ and all domains $D$ compactly contained in $U = f^{-1}(V)$, where $C_D=C_D(U,D,n,p,K,\Sigma)$ is independent of $k$.
\end{lem}
\begin{proof}
    We select $V = \B^n(y_0, 1/e)$ and let $U=f^{-1} V$, which is open since $f$ is continuous by \cite[Corollary 1.2]{kangasniemi2025continuity}. Let $D$ be a domain that is compactly contained in $U$, let $\varepsilon > 0$, and select $\eta\in C^\infty_0(U,[0,1])$ such that  $\eta\equiv 1$ on $D$. 

    Following the methods in \cite{onninen2008mappings}, we apply Lemma \ref{tech lemma} on $f-y_0$ with 
    \begin{equation*}
        \Psi(t)=\frac{1}{2t^{n/2}}\int_0^{\min(t,e^{-1})}\frac{\varphi_\varepsilon(s)}{s\log^n(s^{-1})}ds,\quad {\rm where} \ \varphi_\varepsilon(s)=\frac{1}{1+\varepsilon2^{s^{-1}}}.
    \end{equation*}
    Then we have
    \begin{multline}\label{uniform bouded 1}
        \int_{U}\frac{\eta^n J_f}{|f-y_0|^n\log^n(|f-y_0|^{-2})}\varphi_\varepsilon(|f-y_0|^{2})\\
        \le C(n)\int_{U}\frac{|\nabla \eta|\eta^{n-1} |Df|^{n-1}}{|f-y_0|^{n-1}}\left(\int_0^{\min(|f-y_0|^2,e^{-1})}\frac{\varphi_\varepsilon(s)}{s\log^n(s^{-1})}ds\right).
    \end{multline}
    Arguing  as in \cite[(3.3)]{kangasniemi2025single}, we obtain
    \begin{multline}\label{uniform bouded 2}
        \int_{U}\frac{\eta^n J_f}{|f-y_0|^n\log^n(|f-y_0|^{-2})}\varphi_\varepsilon(|f-y_0|^{2})\\
        \le C(n)\int_{U}\frac{|\nabla \eta|\eta^{n-1}|Df|^{n-1}}{|f-y_0|^{n-1}\log^{n-1}(|f-y_0|^{-2})}\varphi^{\frac{n-1}{n}}_\varepsilon(|f-y_0|^2).
    \end{multline}
    For any $a,b\ge0$ and $\delta>0$, Young's inequality implies that
    \begin{align*}
        a^{n-1}b=\frac{a^{n-1}}{\delta^\frac{n-1}{n}}(\delta^{\frac{n-1}{n}}b)\le \frac{n-1}{n}\frac{a^n}{\delta}+\frac{1}{n}b^n\delta^{n-1}.
    \end{align*}
    We apply an estimate of this type on the right-hand side of \eqref{uniform bouded 2}, and consequently obtain that
    \begin{multline}\label{uniform bounded 3}
            \int_{U}\frac{\eta^nJ_f}{|f-y_0|^n\log^n(|f-y_0|^{-2})}\varphi_\varepsilon(|f-y_0|^{2})\\
            \le\frac{n-1}{n}\int_U\frac{\eta^n|Df|^n}{K|f-y_0|^n\log^n(|f-y_0|^{-2})}\varphi_\varepsilon(|f-y_0|^{2})\\
            \quad+C(n)\int_{U}K^{n-1}|\nabla \eta|^n.
    \end{multline}
    
    Note that the assumed pointwise estimate in \eqref{K,Sigma-quasiregular} yields
    \begin{multline}\label{uniform bounded 4}
             \int_U\frac{\eta^n|Df|^n}{K|f-y_0|^n\log^n(|f-y_0|^{-2})}\varphi_\varepsilon(|f-y_0|^{2})\\
             \le \int_U\frac{\eta^n J_f}{|f-y_0|^n\log^n(|f-y_0|^{-2})}\varphi_\varepsilon(|f-y_0|^{2})\\
             +\int_U\frac{\eta^n \Sigma}{K\log^n{|f-y_0|^{-2}}}\varphi_\varepsilon(|f-y_0|^{2}).
    \end{multline}
    We combine \eqref{uniform bounded 3} and \eqref{uniform bounded 4}, and absorb the term with $|Df|^n/K$ to the left-hand side of \eqref{uniform bounded 4}; this is possible since the function $\varphi_\varepsilon$
    ensures that the integral is finite. The result of this is that
    \begin{multline}\label{uniform bounded 5}
        \int_U\frac{\eta^n|Df|^n}{K|f-y_0|^n\log^n(|f-y_0|^{-2})}\varphi_\varepsilon(|f-y_0|^{2})\\
        \le C(n)\int_{U}\frac{\eta^n \Sigma}{K\log^n|f-y_0|^{-2}}\varphi_\varepsilon(|f-y_0|^{2})+C(n)\int_{U}K^{n-1}|\nabla\eta|^n.
    \end{multline}
    Since \eqref{uniform bounded 5} no longer has any Jacobians in it, all terms in the estimate are non-negative. Hence, we may let $\varepsilon\to0$ and use monotone convergence to obtain
    \begin{multline}\label{uniform bounded 6}
         \int_U
            \frac{\eta^n|Df|^n}{K|f-y_0|^n\log^n(|f-y_0|^{-2})}\\ 
         \le C(n)\int_{U}\frac{\eta^n\Sigma}{K\log^n(|f-y_0|^{-2})}+C(n)\int_{U}K^{n-1}|\nabla\eta|^n.
    \end{multline}
    
    We then note the pointwise estimate
    \begin{equation*}
        |\nabla u_k|\le\frac{|\nabla|f-y_0||}{|f-y_0|\log(|f-y_0|^{-1})}\le \frac{2 |Df|}{|f-y_0|\log(|f-y_0|^{-2})}
    \end{equation*}
    that holds a.e. in $U$. Thus, H\"older's inequality yields that
    \begin{multline}\label{uniform bounded 7}
        \int_D|\nabla u_k|^\beta 
        \le \int_{D}\frac{2^\beta |Df|^\beta}{|f-y_0|^\beta\log^\beta(|f-y_0|^{-2})}\\
        \le 2^\beta \left(\int_D\frac{|Df|^n}{K|f-y_0|^n\log^n(|f-y_0|^{-2})}\right)^{\frac{\beta}{n}}\left(\int_{D}K^p\right)^{\frac{n-\beta}{n}},
    \end{multline}
    where $\beta=\frac{pn}{p+1}<n$. The last term on the right hand side of \eqref{uniform bounded 7} already only depends on $D$, $K$, $n$, and $p$. For the other term, we note that $\log (\abs{f-y_0}^{-2}) \ge 2$ on $U$, and estimate using $\eqref{uniform bounded 6}$ that
    \begin{multline*}
        \int_D\frac{|Df|^n}{K|f-y_0|^n\log^n(|f-y_0|^{-2})}
        \le \int_U
            \frac{\eta^n|Df|^n}{K|f-y_0|^n\log^n(|f-y_0|^{-2})}\\
        \le C(n) \left( \int_{U}\frac{\eta^n\Sigma}{K\log^n(|f-y_0|^{-2})}+ \int_{U}K^{n-1}|\nabla\eta|^n \right) \\
        \le C(n) \left( \int_{\spt \eta}\frac{\eta^n\Sigma}{2^n K}+ \int_{\spt \eta}K^{n-1}|\nabla\eta|^n \right).
    \end{multline*}
    Since the choice of $\eta$ depended only on $D$ and $U$, the claimed upper bound follows.
\end{proof}

If $u_k=\min(\log\log|f-y_0|^{-1},k)$ as in   Lemma \ref{Duk bounded lemma}, the results in Lemma \ref{Duk bounded lemma} bound the $L^\beta$-norm of $\nabla u_k$, for $\beta=\frac{pn}{p+1}<n$, with an upper bound that does not depend on $k$. Moreover, by \cite[Lemma 3.3]{kangasniemi2025single}, we obtain a $k$-independent bound for $L^\beta$-norm of $u_k$.

\begin{lem}[{\cite[Lemma 3.3]{kangasniemi2025single}}]
    Let $B\subset\mathbb{R}^n$ be a ball, let $u: B\to[0,\infty]$ be measurable, and let $p\in [1,\infty)$. Suppose that for every $k\in\mathbb{Z}_{>0}$, the functions $u_k=\min(u,k)$ are in $W^{1,p}(B)$ and satisfy
    \begin{equation*}
        \int_B |\nabla u_k|^p\le C<\infty,
    \end{equation*}
    where $C$ is independent of $k$. Then either $u\equiv\infty$ a.e. on $B$, or $u\in W^{1,p}(B)$ and 
    \begin{equation}\label{uk uniformly bounded}
        \int_B| u_k|^p\le C'<\infty,
    \end{equation}
    for every $k\in\mathbb{Z}_{>0}$, with $C'$ independent of $k$.
\end{lem}

Now, we may complete the proof of the main result of this section. The proof is similar to that in {\cite[Lemma 3.4]{kangasniemi2025single}}.

\begin{lem}\label{totally disconnected}
    Let $\Omega\subset\mathbb{R}^n$ be a domain, let $y_0\in\mathbb{R}^n$, and let $p>n-1,q>1$ with $p^{-1}+q^{-1}<1$.
    Suppose that $f\in W^{1,n}_{\rm loc}(\Omega,\mathbb{R}^n)$ has a value of finite distortion at $y_0\in\mathbb{R}^n$ with data $(K, \Sigma)$, where $K:\Omega\to[1,\infty)$ and $\Sigma:\Omega\to[0,\infty)$ are measurable functions such that
    \begin{equation}
        K\in L^p_{\rm loc}(\Omega) \quad{\rm and}\quad
        \frac{\Sigma}{K}\in L^{q}_{\rm loc}(\Omega).
    \end{equation}
    Then either $f\equiv y_0$ or $\mathcal{H}^{1}(f^{-1}\{y_0\})=0$. In particular, in the latter case $f^{-1}\{y_0\}$ is totally disconnected.
\end{lem}

\begin{proof}
    By Lemma \ref{Duk bounded lemma} and \cite[Lemma 3.3]{kangasniemi2025single}, we find a neighborhood $V$ of $y_0$ such that for every ball $B$ compactly contained in $f^{-1}V$, we have either $f\equiv y_0$ on $B$ or $\log\log|f-y_0|^{-1}\in W^{1,\beta}(B)$, where $\beta=\frac{pn}{p+1}$. If $\log\log|f-y_0|^{-1}\notin W^{1,\beta}_{\rm loc}(f^ {-1}V)$, then there exists a component $U$ of $f^{-1}V$ such that $f\equiv y_0$ on $U$. By \cite[Corollary 1.2]{kangasniemi2025continuity}, $f$ is continuous and consequently we have that $f(\partial U\cap \Omega)\subset\{y_0\}\subset V$. Since the components of $f^{-1}V$ are open, this is only possible if $U=\Omega$. Hence, we either have $\log\log|f-y_0|^{-1}\in W^{1,n}_{\rm loc}(f^{-1}V)$ or $f\equiv y_0$ on all of $\Omega$.

    Suppose that $\log\log|f-y_0|^{-1}\in W^{1,n}_{\rm loc}(f^{-1}V)$. We wish to prove that $\mathcal{H}^{1}(f^{-1}\{y_0\})=0$. Let $x\in f^{-1}\{y_0\}$, and let $B=B^n(x,r)$ be such that $3B=B^n(x,3r)$ is compactly contained in $f^{-1}V$. Select $\eta\in C^\infty_0(3B)$ such that $\eta\ge0$ and $\eta=1$ on $2B$, and we define $v_k=k^{-1}\eta\min(\log\log|f-y_0|^{-1},k)$. Then, for $\beta=\frac{pn}{p+1}\in(n-1,n)$, every $v_k$ is a compactly supported $W^{1,\beta}$-function on $B$ such that $v_k\equiv 1$ in a neighborhood of $f^{-1}\{y_0\}\cap\overline{B}$. Consequently, the functions $v_k$ are admissible for the $\beta$-capacity of the condenser ($f^{-1}\{y_0\}\cap\overline{B},3B$); see \cite[Section 2]{heinonen2018nonlinear} for an exposition on capacity. Since the $L^\beta$-norms of $\nabla v_k$ tend to zero by \eqref{Duk uniformly bounded} and \eqref{uk uniformly bounded}, we have ${\rm Cap}_\beta(f^{-1}\{y_0\}\cap\overline{B},3B)$=0. By \cite[Theorem 2.26]{heinonen2018nonlinear}, the Hausdorff dimension of $f^{-1}\{y_0\}\cap\overline{B}$ is at most $n-\beta<1$, and consequently $\mathcal{H}^{1}(f^{-1}\{y_0\}\cap\overline{B})=0$. The claim $\mathcal{H}^{1}(f^{-1}\{y_0\})=0$ follows by considering a countable cover of $f^{-1}\{y_0\}$ by such $B$, and the proof is complete. 
\end{proof}

\section{Positivity of the degree}
We then connect the definition of topological degree to the setting of Theorem \ref{open and discrete thm}.

\begin{defn}
    Let $\Omega\subset\mathbb{R}^n$ be a domain, let $y_0\in\mathbb{R}^n$, and let $p>n-1,q>1$ with $p^{-1}+q^{-1}<1$. Let
    $f\in W^{1,n}_{\rm loc}(\Omega,\mathbb{R}^n)$ be a non-constant map that satisfies \eqref{K,Sigma-quasiregular} for a.e. $x\in\Omega$, where $K:\Omega\to[1,\infty)$ and $\Sigma:\Omega\to[0,\infty)$ are measurable functions that satisfy conditions \eqref{K,Sigma condition}. Note that $f$ satisfies the Lusin (N) -property by Theorem \ref{thm:lusin_N}.

    We suppose that $U$ is a connected component of $f^{-1}\mathbb{B}^n(y_0,\varepsilon)$ for some $\varepsilon>0$ such that $\overline{U}$ is a compact subset of $\Omega$. Note that if $x\in f^{-1}\{y_0\}$ and $U$ is the $x$-component of $f^{-1}\mathbb{B}^n(y_0,\varepsilon)$ for a small enough $\varepsilon$, then this property holds by \cite[Lemma 4.1]{kangasniemi2025single}. In particular, $U$ is an open set such that $f$ is well defined on $U$ and $f\partial U\subset\partial\mathbb{B}^n(y_0,\varepsilon)$. Thus, the topological degree $\deg(f,y,U)$ is well defined for all points $y\in\mathbb{R}^n\setminus
    \partial \mathbb{B}^n(y_0,\varepsilon)\subset \mathbb{R}^n\setminus f\partial U$. Since $\mathbb{B}^n(y_0,\varepsilon)$ is connected and $\deg(f,y,U)$ is locally constant, $\deg(f,y,U)$ is constant-valued on $\mathbb{B}^n(y_0,\varepsilon)$. For convenience, we call this constant value of $\deg(f,y,U)$ the {\em degree of $f$ with respect to $U$}, and denote it $\deg(f,U)$.
\end{defn}

Now, under the assumptions of Theorem \ref{open and discrete thm}, if $f(x)=y_0$, then for small enough
$\varepsilon>0$ we have a well-defined degree $\deg(f,U)$ for the $x$ -component $U$ of $f^{-1}\mathbb{B}^n(y_0,\varepsilon)$
by \cite[Lemma 4.1]{kangasniemi2025single}. Our objective in this section is to show that, if $\varepsilon$ is small enough,
then this degree is positive.

\subsection{Non-negative values}
We start with a weaker result, which is our counterpart to
\cite[Lemma 5.2]{kangasniemi2025single}. By adapting the arguments therein, we show the non-negativity of the degree.

\begin{lem}\label{non negative lemma}
    Let $\Omega\subset\mathbb{R}^n$ be a domain, let $y_0\in\mathbb{R}^n$, and let $p>n-1,q>1$ with $p^{-1}+q^{-1}<1$.
    Suppose that $f\in W^{1,n}_{\rm loc}(\Omega,\mathbb{R}^n)$ has a value of finite distortion at $y_0\in\mathbb{R}^n$ with data $(K, \Sigma)$, where $K:\Omega\to[1,\infty)$ and $\Sigma:\Omega\to[0,\infty)$ are measurable functions such that
    \begin{equation}
        K\in L^p(\Omega) \quad{\rm and}\quad
        \frac{\Sigma}{K}\in L^{q}(\Omega).
    \end{equation}
    Suppose that $U$ is a non-empty component of $f^{-1}\mathbb{B}^n(y_0,\varepsilon)$, where $\varepsilon>0$ is small enough that $\overline{U}\subset\Omega$. Then there exists $c=c(n,q,K,\Sigma,\Omega)>0$ such that $\deg(f,U)\ge 0$ if $m_n(U)<c$.
\end{lem}

\begin{proof}
    Suppose that $\deg(f,U)<0$. Then $\deg(f,U)\le -1$, and \eqref{degree formula} yields that 
    \begin{equation*}
        \int_{U}J_f\le -\omega_n\varepsilon^n.
    \end{equation*}
    In particular, if we use $J_f^+$ and $J_f^-$ to  denote the positive and negative part of the Jacobian, respectively, then we have
    \begin{equation*}
        \int_{U}J_f^-\ge\omega_n\varepsilon^n.
    \end{equation*}
    However, the distortion inequality \eqref{K,Sigma-distortion} can be rewritten as $|Df|^n+KJ_f^-\le KJ_f^+|f-y_0|^n\Sigma$. Since $J_f^+$ vanishes where $J_f^->0$, we hence have $J_f^-\le K^{-1}|f-y_0|^n\Sigma$. Now we may estimate
    \begin{equation*}
        \omega_n\varepsilon^n\le
        \int_{U}J_f^-\le \int_{U}\frac{|f-y_0|^n\Sigma}{K}\le \varepsilon^n\left\|\frac{\Sigma}{K}\right\|_{L^q(\Omega)}[m_n(U)]^{\frac{q-1}{q}}.
    \end{equation*}
    Hence, we may only have a negative $\deg(f,U)$ if 
    \begin{equation*}
        m_n(U)\ge\left(\omega_n\left\|\frac{\Sigma}{K}\right\|^{-1}_{L^q(\Omega)}\right)^{\frac{q}{q-1}},
    \end{equation*}
    where this lower bound is interpreted as $\infty$ if $\Sigma$ is identically zero.
\end{proof}

\subsection{Zero values} 
We now know that negative values of $\deg(f,U)$ are impossible for small $U$. Next, we wish to exclude the possibility that $\deg(f,U)$ vanishes for small $U$. The ideas and technical tools come from \cite{kangasniemi2025single} coupled with Proposition \ref{prop:Linfty_generalization}.

We will repeat the following long list of assumptions in multiple lemmas, so we record it once here.

\begin{setting}\label{set:long_list_of_assumptions}
    Let $\Omega\subset\mathbb{R}^n$ be a domain, let $y_0\in\mathbb{R}^n$, and let $p>n-1,q>1$ with $p^{-1}+q^{-1}<1$.
    Suppose that
    $f\in W^{1,n}_{\rm loc}(\Omega,\mathbb{R}^n)$ has a value of finite distortion at $y_0\in\mathbb{R}^n$ with data $(K,\Sigma)$, where $K:\Omega\to[1,\infty)$ and $\Sigma:\Omega\to[0,\infty)$ are measurable functions such that
    \begin{equation}
        K\in L^p(\Omega) \quad{\rm and}\quad
        \frac{\Sigma}{K}\in L^{q}(\Omega).
    \end{equation}
    Suppose that $U$ is a non-empty component of $f^{-1}\mathbb{B}^n(y_0,\varepsilon)$, where $\varepsilon>0$ is small enough that $\overline{U}\subset\Omega$ and $m_n(U)<c$ with $c=c(n,q,K,\Sigma,\Omega)>0$ given by Lemma \ref{non negative lemma}.
\end{setting}

The following lemma is a counterpart to \cite[Lemma 5.3-5.4]{kangasniemi2025single}. The proof is similar, though we regardless provide a detailed proof for the convenience of the reader. 

\begin{lem}\label{zero Jacobian lemma}
    Suppose we are under Setting \ref{set:long_list_of_assumptions}.
    If ${\rm deg}(f,U)=0$, then
    \begin{equation}\label{zero Jacobian}
        \int_{U}\frac{J_f}{|f-y_0|^n}=0\quad{\rm and}\quad \int_U\frac{|Df|^n}{K|f-y_0|^n}=\int_U\frac{\Sigma}{K}<\infty.
    \end{equation}
\end{lem}

\begin{proof}
    We first prove that for every $r\in (0,\varepsilon)$, 
    \begin{equation}\label{J_f annulus=0}
        \int_{U\cap \{x\in\mathbb{R}^n: |f-y_0|<r\}}J_f=0.
    \end{equation}
    Let $U_i$, $i\in I$, be disjoint components of $U\cap f^{-1}\mathbb{B}^n(y_0,r)$. Since $U_i\subset U$ and since $m_n(U)<c$, by Lemma \ref{non negative lemma}, we have $\deg (f,U_i)\ge 0$ for every $i\in I$. Moreover, by the additivity and excision properties of the degree, see parts (2) and (3) of Lemma \ref{basic degree property}, we have 
    \begin{equation*}
        \sum_i\deg(f,U_i)=\deg(f,U)=0.
    \end{equation*}
    Hence, for every $i\in I$ we have 
    \begin{equation}\label{eq:subcomponent_zero_degree}
        \deg(f,U_i)=0, 
    \end{equation}
    and therefore by \eqref{degree formula} the integral of $J_f$ over every $U_i$ vanishes. The formula \eqref{J_f annulus=0} hence follows.

    Similarly as in the proof of Lemma \ref{non negative lemma}, the equation \eqref{K,Sigma-distortion} implies that
    \begin{equation*}
        \int_U\frac{J_f^-}{|f-y_0|^n}\le\int_U\frac{\Sigma}{K}<\infty.
    \end{equation*}
    Now, if we denote $U_r=U\cap f^{-1}\mathbb{B}^n(y_0,r)$ for all $r\in (0,\varepsilon)$, \eqref{J_f annulus=0} yields that
    \begin{equation*}
        \int_{U_r}J_f^+=\int_{U_r}J_f^-.
    \end{equation*}
    By multiplying the above equation by $nr^{-n-1}$, integrating over $(0,\varepsilon)$ and switching the order of integrals with Fubini's theorem, we get
    \begin{equation*}
        \int_UJ_f^+(x)\int_{|f(x)-y_0|}^\varepsilon nr^{-n-1}drdx=
        \int_UJ_f^-(x)\int_{|f(x)-y_0|}^\varepsilon nr^{-n-1}drdx.
    \end{equation*}
    By evaluating the inner integral, we have 
    \begin{equation*}
        \int_U\left(\frac{J_f^+}{|f-y_0|^n}-\frac{J_f^+}{\varepsilon^n}\right)= \int_U\left(\frac{J_f^-}{|f-y_0|^n}-\frac{J_f^-}{\varepsilon^n}\right),
    \end{equation*}
    and again using the fact that the integral of $J_f$ over $U$ is 0 yields
    \begin{equation*}
        \int_U\frac{J_f^+}{|f-y_0|^n}=\int_U\frac{J_f^-}{|f-y_0|^n}<\infty.
    \end{equation*}
    Then \eqref{K,Sigma-distortion} immediately gives that
    \begin{equation*}
        \int_U\frac{|Df|^n}{K|f-y_0|^n}=\int_U\frac{\Sigma}{K}<\infty.
    \end{equation*}
\end{proof}
    
We then consider the spherical logarithm map $g$ associated with $f$, as defined in \eqref{spherical-log- formula}, and use Lemma~\ref{zero Jacobian lemma} to obtain a vanishing property of its Jacobian.

\begin{lem}\label{lem: Integral of J_g on annulus}
    Suppose we are under Setting \ref{set:long_list_of_assumptions}. If ${\rm deg}(f,U)=0$ and $g$ is defined as in \eqref{spherical-log- formula}, then we have $g\in W^{1,\beta}(U,\mathbb{R}\times\mathbb{S}^{n-1})$ for some $\beta \in (n-1, n)$, the integral of $J_g$ over $U$ vanishes, and for every $0<t<\varepsilon$, we have
    \begin{equation}\label{Jacobian over annulus}
        \int_{\{x\in U:\ |f-y_0|<t\}}J_g=0.
    \end{equation}
\end{lem}

\begin{proof}
    By H\"older's inequality and \eqref{zero Jacobian}, it follows that 
    \begin{equation*}
        \int_U\frac{|Df|^\beta}{|f-y_0|^\beta}\le\left(\int_U\frac{|Df|^n}{K|f-y_0|^n}\right)^{\frac{\beta}{n}}\left(\int_UK^p\right)^{\frac{n-\beta}{n}}<\infty,
    \end{equation*}
    where $\beta=\frac{np}{p+1}\in(n-1,n)$. By Lemma \ref{lem: spherical-log}, we have $g\in W^{1,\beta}(U,\R\times\mathbb{S}^{n-1})$ and 
    \begin{equation*}
        |Dg|=\frac{|Df|}{|f-y_0|}\quad{\rm and}\quad J_g=\frac{J_f}{|f-y_0|^n}
    \end{equation*}
    for a.e. $x\in U$. 
    
    Let $0<t<\varepsilon$, and let $U_i \subset U$, $i \in I$, be the connected components of $U \cap f^{-1}B(y_0,t)$. For every $i \in I$, we have $m_n(U_i) \le m_n(U) < c$, and by the argument leading up to \eqref{eq:subcomponent_zero_degree}, we also observe that $\deg (f, U_i) = 0$. Thus, by applying part \eqref{zero Jacobian} of Lemma \ref{zero Jacobian lemma} on $U_i$, we have
    \[
        \int_{U_i}J_g=0.
    \]
    Since $U_i$ are pairwise disjoint and since $J_g$ is integrable over $U$, we may hence add up the above integral identities over $i \in I$, which yields \eqref{Jacobian over annulus}. 
\end{proof}

We then combine the previous lemma with Proposition \ref{prop:Linfty_generalization} 
to show a local essential boundedness result for $\log |f-y_0|$.

\begin{lem}\label{lem: varphi is bounded}
    Suppose we are under Setting \ref{set:long_list_of_assumptions}. If $\deg(f,U)=0$, then we have
    \begin{equation*}
        \|\varphi\|^n_{L^\infty(U)}\le C(n,p,q)\|K\|_{L^p(U)}\left\|\frac{\Sigma}{K}\right\|_{L^q(U)}[m_n(U)]^{1-p^{-1}-q^{-1}},
    \end{equation*}
    with $\varphi$ given by 
    \begin{equation}\label{eq: varphi}
        \varphi(x):=\left\{
        \begin{array}{ll}
            -g_\R(x)-\log\frac{2}{\varepsilon}, \quad &x\in U',\\ 
            0,\quad &x\in U \setminus U', 
        \end{array}
        \right.
    \end{equation}
    where $U'$ is the $x_0$-component of $f^{-1}\mathbb{B}^n(y_0,2^{-1}\varepsilon)$ and $g$ is the spherical logarithm \eqref{spherical-log- formula} of $f$ at $y_0$.
\end{lem}

\begin{proof}
    By Lemma \ref{lem: Integral of J_g on annulus} and the fact that $g_\R=-\log\frac{2}{\varepsilon}$ on $\partial U'$, we have $\varphi=0$ on $\partial U'$ and $\varphi\in W^{1,\beta}_{0}(U)$ for $\beta=\frac{pn}{p+1} \in (n-1, n)$. Then, at a.e. $x\in U'$, by \eqref{eq: spherical-log-distortion}, we have $|\nabla\varphi(x)|^n=|\nabla g_\mathbb{R}(x)|^n\le|Dg(x)|^n\le K(x)J_g(x)+\Sigma(x)$. On the other hand, at a.e. $x\in U\setminus U'$, $|\nabla\varphi(x)|=0\le|Dg(x)|^n\le K(x)J_g(x)+\Sigma(x)$. In conclusion, we have
    \begin{equation}\label{eq: nabla-varphi-J_g}
        |\nabla\varphi(x)|^n\le K(x)J_g(x)+\Sigma(x)
    \end{equation}
    for a.e. $x\in U$.

    Thus, we have condition \eqref{eq:Linfty_generalization_dist} of Proposition \ref{prop:Linfty_generalization} with $\varphi \in W^{1, \beta}_0(U)$, $K \in L^p(U)$, $A = \Sigma/K \in L^q(U)$, and $J = J_g$. Moreover, by Lemma \ref{lem: Integral of J_g on annulus}, we have $J_g \in L^1(U)$, and if $t > 0$, then
    \[
        \int_{\varphi^{-1}(t, \infty)} J_g = 0
    \]
    since $\varphi^{-1}(t, \infty) = \{ x \in U : \abs{f - y_0} < s\} \setminus f^{-1}\{y_0\}$ for some $s \in (0, \eps/2)$; note here that $f^{-1}\{y_0\}$ is a null-set by Lemma \ref{totally disconnected}. Thus, condition \eqref{eq:Linfty_generalization_integral} also holds, and Proposition \ref{prop:Linfty_generalization} hence yields the claimed estimate.
\end{proof}

By combining Lemma~\ref{lem: varphi is bounded} with an appropriate choice of $U$, we obtain the desired positivity of $\deg(f,U)$.

\begin{lem}\label{lem: posotive degree}
    Suppose we are under Setting \ref{set:long_list_of_assumptions}. Then $\deg(f,U)>0$.
\end{lem}

\begin{proof}
    Suppose towards contradiction that $\deg(f,U)=0$. Lemma \ref{lem: varphi is bounded} then implies that the $L^\infty$-norm of $\log |f-y_0|$ is finite, that is, $y_0\notin f(U)$, which contradicts $x_0\in f^{-1}\{y_0\}\cap U$. The proof is hence complete.
\end{proof}

\section{The single-value Reshetnyak's theorem}
We now have the essential ingredient required to prove Theorem \ref{open and discrete thm}, namely Lemma \ref{lem: posotive degree}. The remaining parts of the proof are essentially identical to those in \cite[Section 6]{kangasniemi2025continuity}, though we regardless recall those steps for the convenience of the reader.

We begin by proving discreteness, which we require in order to have
a well-defined local index.

\begin{lem}\label{lem: discrete}
    Let $\Omega\subset\mathbb{R}^n$ be a domain, let $y_0\in\mathbb{R}^n$, and let $p>n-1,q>1$ with $p^{-1}+q^{-1}<1$. Suppose that
    $f\in W^{1,n}_{\rm loc}(\Omega,\mathbb{R}^n)$ has a value of finite distortion at $y_0\in\mathbb{R}^n$ with data $(K,\Sigma)$, where $K:\Omega\to[1,\infty)$ and $\Sigma:\Omega\to[0,\infty)$ are measurable functions such that
    \begin{equation*}
        K\in L^p_{\rm loc}(\Omega) \quad{\rm and}\quad
        \frac{\Sigma}{K}\in L^{q}_{\rm loc}(\Omega).
    \end{equation*}
    Then $f^{-1}\{y_0\}$ is a discrete subset of $\Omega$.
\end{lem}

\begin{proof}
    We wish to show that every $x\in f^{-1}\{y_0\}$ has a neighborhood that does not meet the rest of $ f^{-1}\{y_0\}$. For this, we suppose towards contradiction that every neighborhood of a given $x_0\in f^{-1}\{y_0\}$ meets $f^{-1}\{y_0\}\setminus 
    \{x_0\}$. By restricting $f$ to a smaller neighborhood of $x_0$ if necessary, we may assume that $K \in L^p(\Omega)$ and $\Sigma/K \in L^q(\Omega)$. By \cite[Lemma 4.1 and Corollary 4.3]{kangasniemi2025continuity}, we may select an $\varepsilon_0>0$
    such that the $x_0$-component $U_0$ of $f^{-1}\mathbb{B}^n(x_0, \varepsilon_0)$ is compactly contained in $\Omega$, and
    $m_n(U_0)<c$, with $c$ given by Lemma \ref{non negative lemma}. Then the degree $\deg(f,U_0)$ is a finite
    integer; note that the finiteness can be immediately seen from (2.1),
    since $J_f$ is integrable on $U_0$.

    By our counterassumption, there exists an $x_1\in U_0 \cap (f^{-1}\{y_0\} \setminus \{x_0\})$. By \cite[Lemma 4.2]{kangasniemi2025continuity}, we may select an $\varepsilon_1\in(0,\varepsilon_0)$ such that if $U_1$ is the $x_0$-component of $f^{-1}\mathbb{B}^n(y_0,\varepsilon_1)$, then $x_1\notin U_1$. We let $U_{1,i}$ be the other components of $f^{-1}\mathbb{B}^n(y_0,\varepsilon_1)$ contained in $U_0$, where $i\in I_1$. Then by Lemma \ref{non negative lemma}, we have $\deg(f,U_{1,i})\ge0$ for all $i\in I_1$, since $m_n(U_{1,i})\le m_n(U_0)<c$. Moreover, since $x_1$ is necessarily in one of the sets $U_{1,i}$, we must have $\deg(f,U_{1,i})>0$ for at least one $i\in I_1$ by Lemma \ref{lem: posotive degree}.

    Now, we may use the excision and additivity properties of the degree, see Lemma \ref{basic degree property}, to conclude that
    \begin{equation*}
        \deg(f,U_0)=\deg(f,U_1)+\sum_{i\in I_1}\deg(f,U_{1,i})>\deg(f,U_1).
    \end{equation*}
    We can then repeat this procedure to find $U_2$, $U_3$, ... such that $\deg(f,U_1)>\deg(f,U_2)>\cdots$. This is, however, a contradiction, as all degrees $\deg(f,U_i)$ are positive integers by Lemma \ref{lem: posotive degree}. Hence, we conclude that our claim of discreteness holds.
\end{proof}

Under the assumptions of Theorem \ref{open and discrete thm} the discreteness of $f^{-1}\{y_0\}$ implies that the local index $i(x,f)$ is well defined for $x\in f^{-1}\{y_0\}$. Moreover, the sense-preserving part follows almost immediately from our results.

\begin{cor}\label{cor: positive local index}
    Let $\Omega\subset\mathbb{R}^n$ be a domain, let $y_0\in\mathbb{R}^n$, and let $p>n-1,q>1$ with $p^{-1}+q^{-1}<1$. Suppose that
    $f\in W^{1,n}_{\rm loc}(\Omega,\mathbb{R}^n)$ has a value of finite distortion at $y_0\in\mathbb{R}^n$ with data $(K,\Sigma)$, where $K:\Omega\to[1,\infty)$ and $\Sigma:\Omega\to[0,\infty)$ are measurable functions such that
    \begin{equation*}
        K\in L^p_{\rm loc}(\Omega) \quad{\rm and}\quad
        \frac{\Sigma}{K}\in L^{q}_{\rm loc}(\Omega).
    \end{equation*}
    Then for every $x_0\in f^{-1}\{y_0\}$, the local index $i(x_0,f)$ is well-defined and we have $i(x_0,f)>0$.
\end{cor}

\begin{proof}
    By Lemmas \cite[Lemma 4.1]{kangasniemi2025continuity} and \cite[Lemma 4.2]{kangasniemi2025continuity}, along with the discreteness of $f^{-1}\{y_0\}$ shown in
    Lemma \ref{lem: discrete}, there exists an $\varepsilon_0 > 0$ such that for all $\varepsilon \in (0, \varepsilon_0]$, the $x_0$-component $U_\varepsilon$ of
    $f^{-1}\mathbb{B}^n(y_0, \varepsilon)$ is compactly contained in $\Omega$ and satisfies $U_\varepsilon\cap f^{-1}\{y_0\}= \{x_0\}$. Consequently, $i(x_0, f)=\deg(f,U_\varepsilon)$ for any such $\varepsilon$. Hence, by applying Lemma \ref{lem: posotive degree} with $\Omega = U_{\eps_0}$ and Corollary
    \cite[Corollary 4.3]{kangasniemi2025continuity}, it follows that $i(x_0,f)>0$.
\end{proof}

The final property to be deduced is the local openness property.
\begin{lem}\label{lem:open}
    Let $\Omega\subset\mathbb{R}^n$ be a domain, let $y_0\in\mathbb{R}^n$, and let $p>n-1,q>1$ with $p^{-1}+q^{-1}<1$. Suppose that $f\in W^{1,n}_{\rm loc}(\Omega,\mathbb{R}^n)$ has a value of finite distortion at $y_0\in\mathbb{R}^n$ with data $(K,\Sigma)$, where $K:\Omega\to[1,\infty)$ and $\Sigma:\Omega\to[0,\infty)$ are measurable functions such that
    \begin{equation*}
        K\in L^p_{\rm loc}(\Omega) \quad{\rm and}\quad
        \frac{\Sigma}{K}\in L^{q}_{\rm loc}(\Omega).
    \end{equation*}
    Then for every $x_0 \in f^{-1}\{y_0\}$ and every neighborhood $V$ of $x_0$, we have $y_0\in{\rm int}\ f(V)$.
\end{lem}

\begin{proof}
    Suppose towards contradiction that $V$ is an open set containing $x_0\in f^{-1}\{y_0\}$ and that $y_0\notin {\rm int}\ f(V)$. For all $\varepsilon>0$, we use $U_\varepsilon$ to denote the $x_0$-component of $f^{-1}\mathbb{B}^n(y_0,\varepsilon)$. By \cite[Lemmas 4.1 and 4.2]{kangasniemi2025continuity}, we may again select an $\varepsilon_0> 0$ such that $U_{\varepsilon_0}$ is a compact subset of $\Omega$ and $U_{\varepsilon_0}\cap f^{-1}\{y_0\}=\{x_0\}$.

    Let $\varepsilon<\varepsilon_0$, in which case $U_\varepsilon$ is a compact subset of $\Omega$ and $U_\varepsilon\cap f^{-1}\{y_0\}=\{x_0\}$. By Corollary \ref{cor: positive local index} we have $i(x_0,f)>0$, so the definition of local index implies that $\deg(f,U_\varepsilon)>0$. In particular, we have $\deg(f,y,U_\varepsilon)>0$ for all $y\in \mathbb{B}^n(y_0,\varepsilon)$, which by part \eqref{enum:degree_vanishing} of Lemma \ref{basic degree property} implies that $f(U_\varepsilon)= \mathbb{B}^n(y_0,\varepsilon)$.
  
    By our counterassumption, $f(V)$ does not contain $\mathbb{B}^n(y_0,\varepsilon)= f(U_\varepsilon)$. We must
    hence have that $U_\varepsilon\setminus V$ is non-empty. Now, $\overline{U_\varepsilon}\setminus V $ with $\varepsilon\in(0,\varepsilon_0)$ form a decreasing family of non-empty compact subsets of $\Omega$. Hence, their intersection is non-empty. However, this is a contradiction, since
    \begin{align*}
        \bigcap_{\varepsilon\in(0,\varepsilon_0)}(\overline{U_\varepsilon}\setminus V)=\left(\bigcap_{\varepsilon\in(0,\varepsilon_0)}\overline{U_\varepsilon}\right)\setminus V=\{x_0\}\setminus V=\emptyset.
    \end{align*}
    Our original claim therefore holds.
\end{proof}

Thus, by combining Lemma \ref{lem: discrete}, Corollary \ref{cor: positive local index}, and Lemma \ref{lem:open}, Theorem \ref{open and discrete thm} immediately follows. 

\section{Proofs of the Liouville-type theorems}

We begin by recalling a classical version of the Caccioppoli inequality. The statement in fact follows as a special case of Lemma \ref{tech lemma}, though there is a simpler proof via a basic integration by parts. 

\begin{lem}\label{lem: Caccioppoli ineq}
    Let $\Omega\subset\mathbb{R}^n$ be a domain, and let $f\in W^{1,n}_{\rm loc}(\Omega,\mathbb{R}^n)$. Then for every test function $\eta\in C^\infty_0(\Omega)$, we have 
    \begin{equation}\label{ineq: Caccioppoli inequality}
        \left|\int_{\Omega}\eta^nJ_f\right|\le n\int_{\R^n}\eta^{n-1}|\nabla\eta||Df|^{n-1}|f|.
    \end{equation}
\end{lem}

We begin with a lemma about vanishing integrals of Jacobians under specific assumptions. Notably, in the following lemma, we do not yet assume any sort of distortion inequality for our function.
\begin{lem}\label{lem: J_f=0}
    Let $f \in W^{1,n}_\loc(\R^n, \R^n)$, and let $K \colon \R^n \to [1, \infty)$ be a measurable function such that $\mathcal{K}_{n-1} < \infty$, where $\mathcal{K}_{n-1}$ is defined as in \eqref{ineq: K_n}. Suppose that $f \in L^\infty(\R^n, \R^n)$, $J_f^{-} \in L^1(\R^n)$, and $\abs{Df}^n/K \in L^1(\R^n)$. Then $J_f \in L^1(\R^n)$, and
    \begin{equation}\label{eq: int J_f=0}
        \int_{\R^n}J_f=0.
    \end{equation}
\end{lem}
\begin{proof}
    Let $\eta_r\in C^\infty(\R^n,[0,1])$ be such that $\eta_r|_{B^n(0,r)}\equiv1$, $\eta_r|_{\R^n\setminus B^n(0,2r)}\equiv 0$, and $|\nabla\eta_r|\le 2/r$. We begin by showing that
    \begin{equation}\label{eq: eta_n J_f^+ is zero}
        \lim_{r\to\infty}\left|\int_{\R^n}\eta_r^nJ_f\right|=0.
    \end{equation}
    Indeed, by applying the Caccioppoli inequality from Lemma \ref{lem: Caccioppoli ineq} together with H\"older's inequality, we obtain that
    \begin{multline}\label{ineq: eta_rJ_f}
    \left|\int_{\R^n}\eta^n_rJ_f\right|\le n\int_{\R^n}\eta_r^{n-1}|Df|^{n-1}\left|f\right||\nabla\eta_r|\\
    \qquad\le n\left(\int_{{\rm spt}|\nabla\eta_r|}\frac{\eta^n_r|Df|^n}{K}\right)^{\frac{n-1}{n}}\left(\|f\|^n_{L^\infty(\mathbb{R}^n)}
    \dashint_{B^n(0,2r)}K^{n-1}\right)^{\frac{1}{n}}\\
    \qquad\le C\left(\int_{\mathbb{R}^n\setminus B(0,r)}\frac{\eta^n_r|Df|^n}{K}\right)^{\frac{n-1}{n}}\left(
    \dashint_{B^n(0,2r)}K^{n-1}\right)^{\frac{1}{n}},
    \end{multline}
    where the constant $C$ only depends on $n$ and $\|f\|_{L^\infty(\mathbb{R}^n)}$.  Since $|Df|^n/K\in L^1(\R^n)$, the first integral term on the right hand side tends to 0 as $r\to\infty$, while the second term remains bounded. Therefore, the proof of \eqref{eq: eta_n J_f^+ is zero} is complete.

    Now, we use \eqref{ineq: eta_rJ_f} and monotone convergence to argue that
    \begin{multline}\label{ineq: J_f^+}
            \int_{\R^n}J_f^+
            = \lim_{r \to \infty} \int_{\R^n} \eta_r^n J_f^+\\
            \le \limsup_{r \to \infty} 
            \left( \left|\int_{\R^n}\eta_r^nJ_f\right|+\int_{\R^n}\eta_r^nJ_f^- \right)
            = \int_{\R^n}J_f^-.
    \end{multline}
    As $J_f^{-} \in L^1(\R^n)$ by assumption, we thus have $J_f^{+} \in L^1(\R^n)$, and consequently $J_f \in L^1(\R^n)$. Finally, since $|\eta_r^nJ_f|\le |J_f|\in L^1(\mathbb{R}^n)$, by dominated convergence and \eqref{eq: eta_n J_f^+ is zero}, we have
    \begin{equation}
        \int_{\R^n}J_f=\lim_{r\to\infty}\int_{\R^n}\eta_r^nJ_f=0.
    \end{equation}
\end{proof}

Next, we prove that under suitable assumptions, a bounded function satisfying a distortion inequality with defect ends up satisfying the assumptions of Lemma \ref{lem: J_f=0}.

\begin{lem}\label{lem: J_f L^1 J_f=0}
    Let $y_0\in\mathbb{R}^n$. Suppose that $f\in W^{1,n}_{\rm loc}(\R^n,\mathbb{R}^n)$ has a value of finite distortion at $y_0\in\mathbb{R}^n$ with data $(K,\Sigma)$, where $K \colon \R^n\to[1,\infty)$ and $\Sigma\colon \R^n\to[0,\infty)$ are measurable functions such that 
    \begin{equation*}
        \mathcal{K}_{n-1} < \infty
        \quad{\rm and}\quad 
        \frac{\Sigma}{K} \in L^1(\R^n).
    \end{equation*}
    If $f$ is bounded, then $|Df|^n/K\in L^1(\R^n)$, $J_f\in L^1(\R^n)$, and
    \[
        \int_{\R^n} J_f = 0.
    \]
\end{lem}

\begin{proof}
    We again fix cutoff functions $\eta_r\in C^\infty(\R^n,[0,1])$ such that $\eta_r|_{B^n(0,r)}\equiv1$, $\eta_r|_{\R^n\setminus B^n(0,2r)}\equiv 0$, and $|\nabla\eta_r|\le 2/r$.  Now, by using \eqref{K,Sigma-quasiregular}, the Caccioppoli inequality from Lemma \ref{lem: Caccioppoli ineq}, and H\"older's inequality, we obtain
    \begin{multline}\label{ineq: eta_r|Df|^n/K}
        \int_{\R^n}\eta^n_r\frac{|Df|^n}{K}\le\int_{\R^n}\eta^n_rJ_f+\int_{\R^n}\eta_r^n|f-y_0|^n\frac{\Sigma}{K}\\
        \le n\int_{\R^n}(\eta_r|Df|)^{n-1}|f||\nabla\eta_r|+\|f-y_0\|^n_{L^\infty(\R^n)}\int_{\R^n}\frac{\Sigma}{K}\\
        \le C\left[\left(\dashint_{B(0,2r)}K^{n-1}\right)^{\frac{1}{n}}\left(\int_{\R^n}\eta_r^n\frac{|Df|^n}{K}\right)^{\frac{n-1}{n}}+\int_{\R^n}\frac{\Sigma}{K}\right],
    \end{multline}
    where the constant $C$ depends only on $n$, $y_0$, and $\|f\|_{L^\infty(\R^n)}$. Therefore, since $\mathcal{K}_{n-1} < \infty$ and $\Sigma/K \in L^1(\R^n)$, the integral of $\eta^n_r|Df|^n/K$ over $\R^n$ admits a finite upper bound independent of $r$. Letting $r\to\infty$, we obtain via monotone convergence $|Df|^n/K\in L^1(\R^n)$.

    Next, since we have $|Df|^n/K+J^-_f\le J^+_f+(|f-y_0|^n\Sigma)/K$, it follows that the negative part $J_f^{-}$ satisfies a global integrability estimate.
    \begin{equation}\label{ineq: upper bound J_f^-}
        \int_{\R^n}J_f^-\le\int_{\R^n}|f-y_0|^n\frac{\Sigma}{K}
        \le \|f-y_0\|_{L^\infty(\R^n)}^n \int_{\R^n}\frac{\Sigma}{K}.
    \end{equation}
    Thus, $J_f^{-} \in L^1(\R^n)$. Now, the assumptions of Lemma \ref{lem: J_f=0} are satisfied: it follows that $J_f \in L^1(\R^n)$ and the integral of $J_f$ over $\R^n$ vanishes. Thus, the proof is complete.
\end{proof}

Theorem \ref{thm: Liouville FDM case} in fact follows immediately from Lemma \ref{lem: J_f L^1 J_f=0}.

\begin{proof}[Proof of Theorem \ref{thm: Liouville FDM case}]
    Let $\Sigma\equiv 0$, $p=n-1$, and $q=\infty$. Then Lemma \ref{lem: J_f L^1 J_f=0} implies that, for any bounded finite distortion mapping $f \in W^{1,n}_\loc(\R^n, \R^n)$, the integral of $J_f$ over $\R^n$ vanishes. Consequently, since $J_f$ is non-negative, it vanishes a.e.\ on $\R^n$. Thus, by the distortion inequality, $\abs{Df}^n$ vanishes a.e.\ on $\R^n$, and therefore $f$ is constant.
\end{proof}

One of the conclusions of Lemma \ref{lem: J_f L^1 J_f=0} is that the integral of $J_f$ over $\R^n$ vanishes. We now proceed to improve this by showing that the integral of $J_f$ also vanishes over every strict sub-level set of $|f-y_0|$.

\begin{lem}\label{lem: J_f=0 on sub level set}
    Let $y_0\in\mathbb{R}^n$. Suppose that $f\in W^{1,n}_{\rm loc}(\R^n,\mathbb{R}^n)$ has a value of finite distortion at $y_0\in\mathbb{R}^n$ with data $(K,\Sigma)$, where $K \colon \R^n\to[1,\infty)$ and $\Sigma \colon \R^n\to[0,\infty)$ are measurable functions such that 
    \begin{equation*}
        \mathcal{K}_{n-1} < \infty
        \quad{\rm and}\quad 
        \frac{\Sigma}{K} \in L^1(\R^n).
    \end{equation*}
    If $f$ is bounded, then for every $t>0$, we have
    \begin{equation*}
        \int_{\{x\in\R^n : |f-y_0|<t\}}J_f=0.
    \end{equation*}
\end{lem}

\begin{proof}
    Let $t>0$, and let $\psi=\psi_{t}:[0,\infty)\to[0,\infty)$ be the non-decreasing 1-Lipschitz function given by $\psi|_{[0,t]}={\rm id}$ and $\psi|_{[t,\infty)}\equiv t$. 
    Let $h_{t}:\R^n\to\R^n$ be the radial function defined by
    \begin{equation*}
        h_{t}(x)=\psi_{t}(|x|)\frac{x}{|x|},
    \end{equation*}
    noting that $h_{t}$ is 1-Lipschitz. 
    
    Now, let $f_{t}=h_{t}\circ (f-y_0)$. By the chain rule of Lipschitz and Sobolev maps, we have $f_{t}\in W^{1,n}_{\rm loc}(\R^n,\R^n)$, along with $\abs{Df_{t}(x)} \le \abs{Df(x)}$ and $\abs{J_{f_{t}}(x)} \le \abs{J_{f}(x)}$ for a.e.\ $x \in \R^n$; see e.g. \cite[Corollary 3.2]{Ambrosio-dalMaso}. Consequently, since $\abs{Df}^n / K \in L^1(\R^n)$ and $J_f \in L^1(\R^n)$ by Lemma \ref{lem: J_f L^1 J_f=0}, we also have $|Df_t|^n/K\in L^1(\mathbb{R}^n)$ and $J_{f_t} \in L^1(\R^n)$. Since also $\|f_t\|_{L^\infty(\mathbb{R}^n)}\le t < \infty$, we may thus use Lemma \ref{lem: J_f=0} on $f_t$, obtaining that
    \begin{equation*}
        \int_{\{x\in\mathbb{R}^n:\,|f-y_0|<t\}}J_f
        =\int_{\mathbb{R}^n}J_{f_t}=0.
    \end{equation*}
    Thus, the proof is complete.
\end{proof}

We are now ready to prove Theorem \ref{thm:liouville}. The argument we use for this last step is reminiscent of the one used in \cite[Proof of Proposition 1.4]{heikkila2025quasiregular}.

\begin{proof}[Proof of Theorem \ref{thm:liouville}] Suppose  towards contradiction that $f\not\equiv y_0$ and $y_0=f(x_0)$ for some $x_0\in \R^n$. By Theorem \ref{open and discrete thm}, we obtain that $f^{-1}\{y_0\}$ is discrete, and $i(x,f)$ is a positive integer for all $x\in f^{-1}\{y_0\}$. Hence, we may select an open bounded neighborhood $U_0$ of $x_0$ such that $\overline{U_0}\cap f^{-1}\{y_0\}=\{x_0\}$. Next, there exists some small $r_0>0$ such that for all $0<r<r_0$, $\B(y_0,r)\cap f(\partial U_0)=\emptyset$.
By the definitions of the degree and local index that we recalled in Section \ref{sec: degree and index}, we have 
\begin{equation}
    \int_{U_r}J_f=\deg(f,U_r)m_n(\B(y_0,r))=i(x_0,f)\, \omega_n r^n,
\end{equation}
where $U_r:=U_0\cap\{x\in\R^n: |f-y_0|<r\}$ and $\omega_n$ is the volume of the unit ball in $\R^n$. On the other hand, by Lemma \ref{lem: J_f=0 on sub level set}, we obtain that for every $r>0$, the integral of $J_f$ over the sub-level set $\{x\in\R^n: |f-y_0|<r\}$ vanishes. It follows that we have
\begin{equation}
    \int_{U_r}J_f=-\int_{V_r}J_f\le\int_{V_r}J_f^-,
\end{equation}
where $V_r:=\{x\in\R^n: |f-y_0|<r\}\setminus U_0$ and $J_f^-$ again denotes the negative part of the Jacobian determinant $J_f$.

To further estimate this integral, we note that the value of finite distortion at $y_0$ implies that $J_f^-\le |f-y_0|^n\Sigma/K$ a.e. in $\R^n$. Hence, we obtain
\begin{equation}
    \int_{V_r}J_f^{-}\le\int_{V_r}\frac{|f-y_0|^n\Sigma}{K}\le r^n\int_{\{x\in\R^n: |f-y_0|<r\}}\frac{\Sigma}{K}
\end{equation}
since $V_r\subset\{x\in\R^n: |f-y_0|<r\}$. In conclusion, we have that
\begin{equation}
    i(x_0,f)\,\omega_n\le \int_{\{x\in\R^n: |f-y_0|<r\}}\frac{\Sigma}{K}.
\end{equation}

But now, since $\Sigma/K \in L^1(\R^n)$, the convergence properties of the Lebesgue integral for decreasing sequences of sets yield that
\begin{equation}
    i(x_0,f)\,\omega_n
    \le \limsup_{r \to 0} \int_{\{x\in\R^n: |f-y_0|<r\}}\frac{\Sigma}{K}
    = \int_{f^{-1}\{y_0\}} \frac{\Sigma}{K}.
\end{equation}
Since $f^{-1}\{y_0\}$ is discrete, it is a null-set. Thus, we have $i(x_0,f)\,\omega_n\le 0$. This contradicts the fact that $i(x_0,f) >0$, and the proof is hence complete.
\end{proof}

\section{Counterexamples}

We begin by outlining a counterexample to Theorem \ref{open and discrete thm} in the case $K \in L^{n-1}_\loc(\Omega)$ when $n = 2$.

\begin{ex}\label{ex:planar_counterexample}
    We recall that in \cite[Theorem 1.5]{dolevzalova2024mappings}, for some value of $r > 0$, a mapping $h \colon \B^2(0,r)\setminus\{0\} \to \C$ was constructed with the following properties: $h \in W^{1,2}(\B^2(0, r), \C)$, $\abs{Dh}^2 \le K J_h + \Sigma$ with
    \[
        K \in L^1(\B^2(0,r), [1, \infty)) \qquad \text{and} \qquad \Sigma \in L^\infty(\B^2(0, r)), 
    \]
    and $h$ is of the form
    \[
        h(z) = \abs{z} + i \varphi(z)
    \]
    where $\varphi \colon \B^2(0, r) \setminus \{0\} \to [0, \infty)$ is continuous with $\lim_{z \to 0} \varphi(z) = \infty$. We may assume that $r < \pi$ by restricting $h$ if necessary.

    We consider the mapping $f \colon \B^2(0,r) \to \C$ defined by $f(0) = 0$ and
    \[
        f(z) = e^{ih(z)}, \qquad \text{ for } z \in \B^2(0, r) \setminus \{0\}.
    \]
    Since the complex exponential is $1$-Lipschitz on $(-\infty, 0] \times \R$, which contains the image of $ih$, it follows that $f \in W^{1,2}(\B^2(0, r), \C)$. Moreover, by the conformality of the complex exponential, we have $\abs{Df(z)} = \abs{f(z)} \abs{Dh(z)}$ and $J_{f}(z) = \abs{f(z)}^2 J_h(z)$; thus, we have
    \[
        \abs{Df(z)}^2 \le K(z) J_f(z) + \abs{f(z)}^2 \Sigma(z)
    \]
    for a.e.\ $z \in \B^2(0, r)$, where $K \in L^1(\B^2(0, r))$ and $\Sigma \in L^\infty(\B^2(0,r))$. However, since $0 \le \Im(ih) < \pi$ for all $z \in \B^2(0, r)$, the image of $f$ is contained in the closed upper half-plane; since $f(0) = 0$, the map $f$ hence cannot be sense-preserving at $0$, and does not map $\B^2(0, r)$ into a neighborhood of $0$.
\end{ex}

We then provide a counterexample to Theorem \ref{thm:liouville} when the assumption $\mathcal{K}_{n-1} < \infty$ is weakened to $\mathcal{K}_{\beta} < \infty$ for some $\beta < n-1$. This counterexample in fact has $\Sigma \equiv 0$, i.e.\ it is a mapping of finite distortion. The philosophy behind the counterexample bears some resemblance to Ball's counterexample to the openness and discreteness of mappings of finite distortion, see \cite{ball1976convexity, ball1981global} and e.g.\ \cite[Example 3.3]{HenclKoskela_book}.

\begin{example}\label{exam: K_p is finite}
    We begin by constructing an example $f \colon \R^n \to \R^n$ whose first coordinate map $f_1$ is bounded from one side. For this, given $0 < \beta < n-1$, we fix an $\alpha$ with
    \[
        \frac{1}{2} < \alpha < \min \left( 1, \frac{n-1}{2\beta} \right),
    \]
    and define $f \colon \R^n \to \R^n$ by
    \begin{equation}\label{defn: counterexample f}
        f(x_1,...,x_n)=
        \left(|x|\left(1+\frac{x_1}{|x|}
            \right)^\alpha,x_2,...,x_n\right).
    \end{equation}
    A direct computation yields that
    \begin{equation}\label{defn: counterexample Df}
        Df(x)=\left[\left(1+\frac{x_1}{|x|}\right)^{\alpha-1} 
        \left(\alpha e_1+\left(1+(1-\alpha)
            \frac{x_1}{|x|}\right)\frac{x}{|x|}\right),
        e_2,...,e_n\right].
    \end{equation}
    Consequently, we have 
    \begin{equation*}
        J_f=\left(\alpha+\left(1+(1-\alpha)\frac{x_1}{|x|}\right)\frac{x_1}{|x|}\right)\left(1+\frac{x_1}{|x|}\right)^{\alpha-1},
    \end{equation*}
    and the Hilbert-Schmidt norm of $Df$ satisfies
    \begin{equation*}
        \begin{aligned}
        \norm{Df(x)}^2 = (n-1)
        &+\left(\alpha^2+\left(1+(1-\alpha)\frac{x_1}{|x|}\right)^2\right)\left(1+\frac{x_1}{|x|}\right)^{2\alpha-2}\\
        &+2\alpha\frac{x_1}{|x|}\left(1+(1-\alpha)\frac{x_1}{|x|}\right)\left(1+\frac{x_1}{|x|}\right)^{2\alpha-2}.
        \end{aligned}
    \end{equation*}
    We then adopt a coordinate system where $x_1 = r \cos \varphi$ and $(x_2, \dots, x_n) = \sin(\varphi) \theta$, where $(r, \varphi, \theta) \in [0, \infty) \times [0, 2\pi) \times \S^{n-2}$. We then observe that $J_f$ and $\norm{Df}$ only depend on $\varphi$. More precisely,
    \begin{equation*}
        J_f(r,\varphi,\theta)=\left(\alpha+\cos\varphi+(1-\alpha)\cos^2\varphi\right)\left(1+\cos\varphi\right)^{\alpha-1}
    \end{equation*}
    and 
     \begin{equation*}
        \begin{aligned}
        \norm{Df(r, \varphi, \theta)}^2&=(n-1)
        +\left((1+\cos\varphi)^2
        +\alpha^2(1-\cos^2\varphi)\right)\left(1+\cos\varphi\right)^{2\alpha-2}\\
        &=(n-1)+\left((1+\cos\varphi)
        +\alpha^2(1-\cos\varphi)\right)\left(1+\cos\varphi\right)^{2\alpha-1}.
        \end{aligned}
    \end{equation*}
    
    Next, we estimate $J_f(r,\varphi,\theta)$ and $\norm{Df(r,\varphi,\theta)}^2$. For $J_f(r,\varphi,\theta)$, we consider the function $F \colon [0, 2\pi] \to \R$ given by
    \begin{equation*}
        F(\varphi) = \alpha+\cos\varphi+(1-\alpha)\cos^2\varphi = (1+\cos\varphi)+(1-\alpha)\sin^2\varphi.
    \end{equation*}
    From the latter expression, we observe that $F(\varphi) \ge 0$, and $F(\varphi) = 0$ if and only $\varphi=\pi$. Consequently,  $J_f(r,\varphi,\theta) \ne 0$ when $\varphi \ne \pi$. Moreover, using the elementary estimates
    \begin{equation}\label{eq:sine_cosine_asymptotics}
        0 \le \sin^2\varphi \le C |\theta-\pi|^2
        \quad{\rm and}\quad
        C^{-1} |\theta-\pi|^2 \le 1+\cos\varphi
        \le C |\theta-\pi|^2
    \end{equation}
    for $\varphi \in [0, 2\pi]$, we have 
    \begin{equation*}
        J_f(r,\varphi,\theta) \approx |\varphi-\pi|^{2\alpha}, \quad\text{for }  \varphi\in [0, 2\pi],
    \end{equation*}
    with the comparison constants depending only on $\alpha$.
    
    For $\norm{Df(r,\varphi,\theta)}^2$, observe that
    \begin{equation*}
        2\alpha^2 \le (1+\cos\varphi)
        +\alpha^2(1-\cos\varphi) = 1 + \alpha^2 + (1 - \alpha^2) \cos \varphi \le 2.
    \end{equation*}
    Thus, this term is bounded from above and below by positive constants dependent only on $\alpha$. Therefore, by again using \eqref{eq:sine_cosine_asymptotics}, we obtain 
    \begin{equation*}
        \norm{Df(r,\varphi,\theta)}^2 \approx (n-1)+|\varphi-\pi|^{4\alpha-2},\quad\text{for }  \varphi\in [0, 2\pi],
    \end{equation*}
    with the comparison constants depending only on $\alpha$. Since $\alpha > 1/2$, we have $4\alpha - 2 > 0$; thus, we in fact have
    \[
        C^{-1}(n, \alpha) \le \norm{Df(r,\varphi,\theta)}^2 \le C(n, \alpha),
    \]
    which proves that $f$ is Lipschitz, and hence in $W^{1,n}_\loc(\R^n, \R^n)$.
    
    Note that $\abs{Df}$ and $\norm{Df}$ are uniformly comparable with a dimensional constant. Thus, in view of the above estimates for $J_f(\varphi)$ and $\norm{Df(\varphi)}^2$, the distortion $K_f:=|Df|^n/J_f$ satisfies
    \begin{equation}\label{ineq: K varphi}
        K_f(r,\varphi,\theta) \le C(n, \alpha) |\varphi-\pi|^{-2\alpha}
        \quad \text{for all } \varphi \in [0, 2\pi].
    \end{equation}
    
    We now claim that the average integral of $K_f^\beta$ on  balls is finite. Indeed, in our coordinate system $x = (r\cos \varphi, \sin(\varphi)\theta)$, the volume element is $dm_n(x) = r^{n-1} \sin^{n-2} (\varphi) \, dr \, d\varphi \, d\cH^{n-2}(\theta)$. Thus, using \eqref{ineq: K varphi}, we have
    \begin{align*}
        &\dashint_{\B^n(0,R)} K_f^\beta(x)\,dm_n(x)\\
        &\quad\le \frac{C(n, \alpha)}{m_n(\B^n(0,R))}
        \int_{\S^{n-2}} \int_0^R \int_0^{2\pi} \frac{r^{n-1} \sin^{n-2}(\varphi)}{|\varphi-\pi|^{2\alpha\beta}} \, d\varphi dr d\cH^{n-2}(\theta)\\
        &\quad\le C(n, \alpha) \int_0^{2\pi} \frac{\sin^{n-2}(\varphi)}{|\varphi-\pi|^{2\alpha\beta}} \, d\varphi.
    \end{align*}
    Combining this with the upper bound in \eqref{eq:sine_cosine_asymptotics}, it suffices to show that
    \begin{equation}\label{ineq: final estimate of K^p}
        \int_0^{2\pi} |\varphi-\pi|^{n-2-2\beta\alpha} \,d\varphi<\infty,
    \end{equation}
    which holds if and only if $\alpha < (n-1)/2\beta$. As this was one of our assumptions, the integral in \eqref{ineq: final estimate of K^p} is indeed finite. Thus, our $f$ is a continuous mapping of finite distortion in $W^{1,n}_\loc(\R^n, \R^n)$ with $\mathcal{K}_\beta < \infty$, but $f_1(x) \ge 0$ for all $x \in \R^n$.
    
    Finally, using $f$, we construct a non-constant bounded mapping $h\in W^{1,n}_{\rm loc}(\mathbb{R}^n,\mathbb{R}^n)$ of finite distortion such that $\mathcal{K}_\beta < \infty$. For this, let $\tau$ be a M\"obius transformation that maps $\mathbb{H}^n$ onto $\mathbb{B}^n$, noting that $\tau$ is conformal. Then the composition $h:=\tau\circ f:\mathbb{R}^n\to\mathbb{R}^n$ is bounded. Moreover, $f$ is Lipschitz, and $\tau$ is Lipschitz on $\H^n$. Thus, by the chain rule for Lipschitz maps, we have that $h \in W^{1,n}_\loc(\R^n,\R^n)$ and  $Dh(x)=D\tau(f(x))Df(x)$ for a.e. $x \in \R^n$; see again e.g. \cite[Corollary 3.2]{Ambrosio-dalMaso}. In particular,
    \begin{equation*}
        \begin{aligned}
        |Dh(x)|^n&\le |D\tau(f(x))|^n|Df(x)|^n\\
        &\le J_\tau(f(x)) K_f(x) J_f(x)=K_f(x)J_{h}(x)
        \end{aligned}
    \end{equation*}
    for a.e. $x\in\R^n$. Therefore, $h$ is a mapping of finite distortion with distortion $K_h=K_f$, and hence $\mathcal{K}_\beta < \infty$ for $h$, completing the example.
 \end{example}

\subsection*{Acknowledgments}

We thank Kai Rajala for a useful comment regarding the history of the Liouville theorem for mappings of finite distortion.

\bibliographystyle{plain}
\bibliography{Reshetnyak.bib}
\end{document}